\newif\ifTwoColumn
\newif\ifTechReport

\TechReporttrue

\ifTwoColumn
	\documentclass[journal,letterpaper]{ieeetran}
	\usepackage{flushend}
	\IEEEoverridecommandlockouts

	\newenvironment{proof}{
	\begin{IEEEproof}
	}{
	\end{IEEEproof}
	}
\else
	\ifTechReport
		\documentclass[11pt,onecolumn,a4paper]{IEEEtran}
		\usepackage{geometry}
		\usepackage{amsthm}%
		\usepackage{parskip}
	    \makeatletter
		\def\thm@space@setup{%
		  \thm@preskip=\parskip \thm@postskip=0pt
		}
		\makeatother

	\else
		\documentclass[12pt,draftcls,onecolumn,letterpaper]{IEEEtran}
		\usepackage{amsthm}
		\IEEEoverridecommandlockouts
	\fi
\fi
\usepackage{graphicx}
\usepackage[cmex10]{amsmath}
\usepackage{amssymb}
\usepackage{mathtools}
\usepackage{epstopdf}
\usepackage{balance}
\usepackage{accents}
\usepackage{pgfplots}
\graphicspath{{figs/}{./}}

\usepackage{etoolbox}

\usepackage[utf8]{inputenc}

\newtoggle{useBiber}
\toggletrue{useBiber} %

\iftoggle{useBiber}{
\usepackage[backend=biber,style=ieee,url=false,doi=false,isbn=false]{biblatex}

\DeclareFieldFormat{number}{(#1)}
\DeclareFieldFormat{pages}{#1}

\renewbibmacro*{volume+number+eid}{%
  \printfield{volume}%
  (\printfield{number})}

\addbibresource{bibliography.bib}
\AtEveryBibitem{\clearlist{language}}
\AtEveryBibitem{\clearfield{month}}

}{
}

\usepackage{xcolor,calc}

\newtheorem{theorem}{Theorem}
\newtheorem{lemma}{Lemma}
\newtheorem{corollary}[lemma]{Corollary}

\newtheorem{definition}{Definition}

\newtheorem{ass}{Assumption}

\newcommand{\R}{\mathbb{R}}
\newcommand{\N}{\mathbb{N}}
\newcommand{\Nn}[1]{\N_{#1}}
\renewcommand{\Re}{\mathbb{R}}

\newcommand{\nullspace}{\mathcal{N}}

\newcommand{\one}{\mathbf{1}}

\newcommand{\sign}{\mathrm{sgn}}
\renewcommand{\int}{\mathrm{int}}
\newcommand{\bigmid}{\;\middle|\;}

\def\*#1{\mathbf{#1}}

\newcommand{\bm}{\begin{pmatrix}}
\newcommand{\bmend}{\end{pmatrix}}

\newcommand{\sol}[2]{\mathrm{SOL}\left(#1,#2\right)}

\usepackage[hidelinks]{hyperref}

\title{Strong Stationarity Conditions for Optimal Control of Hybrid Systems}
\author{Andreas B. Hempel, Paul J. Goulart, and John Lygeros
\thanks{The research leading to these results has received funding from the European Commission under the project SPEEDD, grant number 619435.}
\thanks{Andreas Hempel and John Lygeros are with the Automatic Control Laboratory, ETH Z\"urich, Physikstrasse 3, 8092 Z\"urich, Switzerland; \texttt{\small \{hempel,lygeros\}@control.ee.ethz.ch}. Paul Goulart is with the Department of Engineering Science, University of Oxford, Parks Road, Oxford, OX1 3PJ, UK; \texttt{\small paul.goulart@eng.ox.ac.uk}.
}
}
\date{\today}                                           %

\begin{document}
\maketitle

\begin{abstract}
We present necessary and sufficient optimality conditions for finite time optimal control problems for a class of hybrid systems described by linear complementarity models. Although these optimal control problems are difficult in general due to the presence of complementarity constraints, we provide a set of structural assumptions  ensuring that the tangent cone of the constraints possesses geometric regularity properties. These imply that the classical Karush-Kuhn-Tucker conditions of nonlinear programming theory are both necessary and sufficient for local optimality, which is not the case for general mathematical programs with complementarity constraints.  We also present sufficient conditions for global optimality.

We proceed to show that the dynamics of every continuous piecewise affine system can be written as the optimizer of a mathematical program which results in a linear complementarity model satisfying our structural assumptions. Hence, our stationarity results apply to a large class of hybrid systems with piecewise affine dynamics. We present simulation results showing the substantial benefits possible from using a nonlinear programming approach to the optimal control problem with complementarity constraints instead of a more traditional mixed-integer formulation.
\end{abstract}

Many dynamical system have both continuous and switching, i.e.\ hybrid, dynamics~\cite{Goebel2009}. Common examples include mechanical systems with impact~\cite{Brogliato1999}, traction models for cars~\cite{Borrelli2006,DiCairano2010}, and biological systems~\cite{Mirollo1990}. The class of \emph{piecewise affine} (PWA) systems is particularly important because it is a natural generalization of linear systems and hence conceptually easy to understand and use for modeling~\cite{Sontag1981}.
It has been shown in~\cite{Heemels2001} that PWA systems are equivalent to a number of other hybrid system model classes such as \emph{linear complementarity} (LC) models~\cite{Heemels2003}, mixed-logical dynamical (MLD)~\cite{Bemporad1999}, and max-min-plus-scaling models~\cite{DeSchutter2004} under certain technical assumptions. Some tools and results developed for one system class may, accordingly, be applied to another, equivalent system class where appropriate. Finite horizon optimal control problems for discrete time hybrid system models have for example been studied in~\cite{Bemporad1999} (PWA / MLD models) and~\cite{DeSchutter2000,DeSchutter2001} (max-min-plus-scaling models).

Here we consider hybrid dynamical systems in LC form:
\begin{subequations}
\label{eq:LCModel}
\begin{align}
x^+ &= A x + B_u u + B_w w + c \label{eq:LCTransition}, \\
0 &\leq E_w w + E_x x + E_u u + e \enspace \perp \enspace w \geq 0, \label{eq:LCComplementarity}
\end{align}
\end{subequations}
where $x \in \R^{n_x}$ is the current system state and $x^+ \in \R^{n_x}$ the successor state resulting from applying a control input $u \in \R^{n_u}$.
Since both sides of \eqref{eq:LCComplementarity} are constrained to be non-negative, the orthogonality relation $\perp$ requires that every component of the \emph{complementarity variable} $w \in \R^{n_w}$ is constrained to be zero if the corresponding component on the left-hand side of~\eqref{eq:LCComplementarity} is non-zero, and vice-versa.
While the complementarity variables $w$ can by ascribed physical meaning with respect to the modeled system in some cases, see~\cite[Ch.\ 11]{Biegler2010},~\cite{Ferris1997}, or~\cite{Vasca2009}, we will consider the general case where they are treated as arbitrary auxiliary variables. The other matrices in~\eqref{eq:LCModel} are assumed to have compatible dimensions, in particular $E_w \in \R^{n_w \times n_w}$ is a square matrix.

For the hybrid system~\eqref{eq:LCModel} we consider finite horizon optimal control problems:
\begin{subequations}
\label{eq:LCCFTOC}
\begin{align}
\min_{\*u,\*x,\*w} \quad  &{\ell}_N(x_N) + \sum_{k=0}^{N-1} {\ell}_k(x_k,u_k) &\label{eq:LCMPCCost} \\
\text{s.t.} \quad %
&x_{k+1} = A x_k + B_u u_k + B_w w_k + c \label{eq:LCMPCDynamicAssignment}  \\
&0 \leq E_w w_k + E_x x_k + E_u u_k + e \enspace \perp \enspace w_k \geq 0 \label{eq:LCMPCComplementarity}
\end{align}
\end{subequations}
Here, $x_k \in \R^{n_x}$ denotes the \emph{predicted} system state $k$ timesteps into the future and the constraints have to be satisfied for all $k = 0, \dots, N-1$.
We denote by $\*u := (u_0, \dots, u_{N-1})$ the collection of the control inputs to be chosen, by $\*x := (x_0, \dots, x_N)$ the resulting state trajectory starting from the initial state $x_0 = x$, and by $\*w := (w_0, \dots, w_{N-1})$ the corresponding trajectory of complementarity variables. We assume the typical case of convex stage costs $\ell_k \colon \R^{n_x} \times \R^{n_u} \to \R$ and terminal cost $\ell_N \colon \R^{n_x} \to \R$ that are independent of $w_k$.
The constraints~\eqref{eq:LCMPCDynamicAssignment} and~\eqref{eq:LCMPCComplementarity} have to hold for all $k = 0, \dots, N-1$. The variables $w_k$ do not appear in the the cost function~\eqref{eq:LCMPCCost} because we consider them to be auxiliary variables without physical meaning.

We deliberately omit constraints on the state and control input from the optimal control problem~\eqref{eq:LCCFTOC} for clarity of exposition.  However, all of the results in this paper extend to linear state-input constraints over the horizon, i.e.\ $(x_k, u_k) \in \Gamma_k$ for polytopes $\Gamma_k \subseteq \R^{n_x} \times \R^{n_u}$.

Due to the complementarity constraints~\eqref{eq:LCMPCComplementarity} that are part of the LC dynamics, the optimal control problem~\eqref{eq:LCCFTOC} is called a \emph{mathematical program with complementarity constraints} (MPCC)~\cite{Luo1996}. For an overview of MPCC solution methods and different applications see~\cite{Biegler2010,Ralph2008}.
MPCCs are nonlinear and nonconvex optimization problems (i.e.\ the complementarity constraint~\eqref{eq:LCMPCComplementarity} can be modeled as a bilinear (in-)equality) and are generally difficult to solve. The main reason for this is that the classical Mangasarian-Fromovitz constraint qualification (MFCQ) of nonlinear programming (NLP) is violated at every feasible point~\cite{Chen1995}. MFCQ is typically assumed in NLP theory to enable the use of Karush-Kuhn-Tucker (KKT) conditions to characterize optimality. It is equivalent to the compactness of the set of Lagrange multipliers at a local optimum~\cite{Gauvin1977}.

In this paper we make a number of structural assumptions on~\eqref{eq:LCModel} to focus on a particular subclass of LC models. Under these assumptions we can prove strong stationarity conditions for the optimal control MPCC~\eqref{eq:LCCFTOC}. In particular, we will show that the classical KKT conditions from NLP theory are both necessary and sufficient for optimality in~\eqref{eq:LCCFTOC}. This is generally not the case for MPCCs where KKT conditions cannot be expected to hold and specialized solution methods are often used~\cite{Luo1996}.

We then show that our results for~\eqref{eq:LCCFTOC} cover optimal control problems for arbitrary continuous PWA systems. In particular, we exploit the fact that continuous PWA functions can be written as the difference of two convex PWA functions~\cite{Kripfganz1987,Hempel2013} to represent PWA system dynamics as the solution to a convex parametric quadratic program (PQP). Further manipulations result in a special LC model~\eqref{eq:LCModel} which satisfies the assumptions we require on~\eqref{eq:LCModel} for our strong stationarity results.
We also present sufficient conditions for a given local optimum to be the only local optimum within a neighborhood around it or even \emph{globally optimal} for~\eqref{eq:LCCFTOC}.

In addition to an illustrative example we present numerical results that indicate solving~\eqref{eq:LCCFTOC} as an NLP can be computationally advantageous  compared to more traditional mixed-integer approaches to solving optimal control problems for PWA systems~\cite{Bemporad1999}. Treating~\eqref{eq:LCCFTOC} as a general NLP instead of an MPCC is only possible due to the strong stationarity results in this paper since typical NLP algorithms attempt to find a solution to the KKT conditions which, as mentioned above, are not suitable optimality conditions for MPCCs. Solving~\eqref{eq:LCCFTOC} as an NLP leads to significantly shorter computation times and much better scaling in the problem dimensions than an MIP approach and often yields globally optimal solutions in numerical experiments. While no a-priori guarantees exist for this, simulation studies for randomly generated systems outline the possible benefits of this approach. Since optimal control problems for hybrid systems are NP hard in general no such guarantees should be expected~\cite{Daafouz2009}.

\textbf{Notation:}
We use superscripts in square brackets to indicate elements of vectors and vector valued functions and matrices, e.g.\ $f^{[j]}$ is the $j^{\text{th}}$ element of $f$ and $M^{[j,:]}$ is the $j^{\text{th}}$ row of $M$. We use the symbols $\wedge$ and $\vee$ to denote the logical operations ``and'' and ``or'', respectively. The Kronecker product between two matrices $A$ and $B$ is denoted $A \otimes B$. The notation $x \perp y$ indicates that the two vector $x,y \in \R^n$ are orthogonal, i.e.\ $x^\top y = 0$.
A bold~$\one$ denotes a vector (or matrix) of ones and $I$ an identity matrix. In case the dimensions are not clear from context we indicate them using subscripts, e.g.\ $\one_{n \times m} \in \R^{n \times m}$ is an $n$-by-$m$ matrix of ones. A positive \mbox{(semi-)}definiteness condition on a symmetric matrix $M=M^\top$ is denoted $M\succ0$ ($M \succeq 0$).
While $\lVert \cdot \rVert$ denotes an arbitrary norm on the respective space, we denote a weighted Euclidean norm using a symmetric matrix $Q \succ 0$ as $\lVert x \rVert_Q := \sqrt{x^\top Q x}$. %
The sign of a scalar $x$ is given by $\sign(x)$.

The  nullspace of a matrix $M$ is given by  $\nullspace(M)$.
The interior of a set $\Omega \subseteq \R^n$ is denoted $\int(\Omega)$ and the cardinality of a set $\gamma \subseteq \N$ by $\lvert \gamma \rvert$. We denote the set of real non-negative numbers $\R_+$ and the index set $\{1, \dots, n\}$ by $\Nn{n}$. All vector-valued inequalities are to be understood component-wise.

The term \emph{mathematical program with equilibrium constraints} (MPEC) is also used regularly in the literature for MPCCs such as~\eqref{eq:LCCFTOC}, although the two problem classes are not entirely equivalent~\cite{Dempe2012}. We will use the term \emph{MPCC} throughout the paper, except for some technical terms, e.g.\ MPEC-Abadie constraint qualification, which have been firmly established in the literature using the MPEC-acronym.

\section{Standing assumptions and main result}
\label{sec:Problem}

Throughout the paper we make  a number of standing assumptions about the LC model~\eqref{eq:LCModel}. The first is a reasonable assumption to guarantee deterministic behavior:
\begin{ass} \label{ass:WellPosedness}
Every complementarity variable $w$ that solves~\eqref{eq:LCComplementarity} for a fixed $(x,u)$ results in the same successor state $x^+$.
\end{ass}
For an in-depth discussion of well-posedness of LC systems and related results we refer to~\cite{Camlibel2001}.

The next is a technical assumption on the structure of the complementarity problem~\eqref{eq:LCComplementarity}:
\begin{ass} \label{ass:DecomposablePSDRankOne}
There exist an $n_b \in \N$ and scalars $n_{w_1}, \dots, n_{w_{n_b}} \in \N$ that partition $\Nn{n_w}$ such that the complementarity problem~\eqref{eq:LCComplementarity} can (possibly after a coordinate transformation) be decomposed into $n_b$ independent complementarity problems
\begin{equation} \label{eq:SubComplementarityProblem}
\forall i \in \Nn{n_b} \colon \quad
0 \leq M_i {w}_i + C_i x + D_i u + e_i \enspace \perp \enspace {w}_i \geq 0,
\end{equation}
where $M_i = M_i^\top \in \R^{n_{w_i} \times n_{w_i}}$ is a symmetric positive semidefinite rank-one matrix, i.e.
\[
\forall i \in \Nn{n_b} \colon \quad M_i = m_i m_i^\top \succeq 0
\]
for some $m_i \in \R^{n_{w_i}}$. This decomposition is minimal in the sense that $m_i^{[j]} \neq 0$ for every $i \in \Nn{n_b}$ and $j \in \Nn{n_{w_i}}$.
\end{ass}

Assumption~\ref{ass:DecomposablePSDRankOne} means that ${E}_w = {E}_w^\top \in \R^{n_w \times n_w}$ is a block-diagonal matrix with the rank one matrices $M_i$ from~\eqref{eq:SubComplementarityProblem} along the diagonal
\begin{equation*}
{E}_w = \begin{pmatrix}
M_{1} & 0 & \dots & 0 \\
0 & M_{2} & \dots & 0 \\
\vdots & & \ddots & \vdots \\
0 & 0 & \dots & M_{{n_b}}
\end{pmatrix} \text{ with }
w = \begin{pmatrix} w_1 \\ \vdots \\ w_{n_b} \end{pmatrix},
\end{equation*}
and the other matrices in~\eqref{eq:LCComplementarity} are given as
\begin{align*}
{E}_x &= \begin{pmatrix}
C_{1} \\
\vdots \\
C_{{n_b}}
\end{pmatrix}, &
{E}_u &= \begin{pmatrix}
D_{1} \\
\vdots \\
D_{{n_b}}
\end{pmatrix}, &
{e} &= \begin{pmatrix}
e_{1} \\
\vdots \\
e_{{n_b}}
\end{pmatrix}.
\end{align*}
The solution set of the complementarity problem~\eqref{eq:LCComplementarity} is simply the cartesian product of the solution sets of~\eqref{eq:SubComplementarityProblem}. %

The requirement that $m_i^{[j]}$ be nonzero is without loss of generality and serves to avoid the presence of spurious degenerate complementarity variables. Assume $m_i^{[j]} = 0$ and $C_i^{[j,:]} x + D_i^{[j,:]} u + e_i^{[j]} = 0$ in~\eqref{eq:SubComplementarityProblem}, then $w_i^{[j]}$ is unbounded above and completely independent from all other complementarity constraints. This means its corresponding column in $B_w$ has to be identically zero to satisfy Assumption~\ref{ass:WellPosedness}. Hence, $w_i^{[j]}$ could simply be eliminated from the model.

Assumption~\ref{ass:DecomposablePSDRankOne} obviously limits what LC models~\eqref{eq:LCModel} we can consider, e.g.\ a consequence of Assumption~\ref{ass:DecomposablePSDRankOne} is that $E_w = E_w^\top \succeq 0$ which is generally not the case. We will prove in Section~\ref{sec:PWASystems} that every continuous PWA system can be written as an LC model in the form~\eqref{eq:LCModel} satisfying Assumption~\ref{ass:DecomposablePSDRankOne}. Hence, while we have limited the number of LC models~\eqref{eq:LCModel} we consider, we still cover all cases where equivalent continuous PWA dynamics exist.

It is easy to prove the following result which provides an easy procedure to verify whether certain LC models satisfy Assumption~\ref{ass:WellPosedness}:

\begin{lemma}[nullspace representation of Assumption~\ref{ass:WellPosedness}] \label{lem:NecessaryConditionWellPosedness}
An LC model~\eqref{eq:LCModel} with $E_w = E_w^\top \succeq 0$, e.g.\ one that satisfies Assumption~\ref{ass:DecomposablePSDRankOne}, satisfies Assumption~\ref{ass:WellPosedness} if
\[
\nullspace\left(E_w\right) \subseteq \nullspace\left( B_w \right).
\]
\end{lemma}
\begin{proof}
Use~\cite[Thm.\ 3.1.7]{Cottle1992}.
\end{proof}

Finally, we make an assumption on the existence of special solutions for the complementarity sub-problems~\eqref{eq:SubComplementarityProblem}:

\begin{ass} \label{ass:NontrivialSolutions}
For every fixed $x \in \R^{n_x}$ and $u \in \R^{n_u}$ and every sub-problem~\eqref{eq:SubComplementarityProblem} with $i \in \Nn{n_b}$ there exists a $j \in \Nn{n_{w_i}}$ such that
\[
C_i^{[j,:]} x + D_i^{[j,:]}u + e_i^{[j]} < 0.
\]
\end{ass}
This assumption ensures that for any $(x,u)$ the solution to~\eqref{eq:SubComplementarityProblem} is non-trivial, i.e.\ $w_i \neq 0$.
Similarly to Assumption~\ref{ass:DecomposablePSDRankOne}, Assumption~\ref{ass:NontrivialSolutions} restricts the LC models~\eqref{eq:LCModel} considered but will be satisfied by LC models obtained from an arbitrary continuous PWA model via our construction described in Section~\ref{sec:PWASystems}.

Failure of MFCQ for~\eqref{eq:LCCFTOC} means that the set of Lagrange multipliers is either unbounded (which can be a numerical problem) or may even be empty at a local optimum. This means it is generally not possible to characterize optimality for MPCCs using the classical KKT conditions. The main result of this paper relieves this problem for the particular optimal control MPCC~\eqref{eq:LCCFTOC}:

\begin{theorem}[strong stationarity conditions] \label{thm:OptimalityConditions}
Let $\*v^* := (\*x^*, \*u^*, \*w^*)$ be feasible for the optimal control problem~\eqref{eq:LCCFTOC} for an LC model satisfying Assumptions~\ref{ass:WellPosedness},~\ref{ass:DecomposablePSDRankOne}, and~\ref{ass:NontrivialSolutions}. Then $\*v^*$ is locally optimal if and only if the classical KKT conditions in the sense of~\cite[Sec. 5.5.3]{Boyd2005},~\cite[Sec. 5.1]{Bertsekas2003} for~\eqref{eq:LCCFTOC} admit a primal-dual solution pair.
\end{theorem}

We will need a number of intermediate results to prove Theorem~\ref{thm:OptimalityConditions}, so defer the proof to Section~\ref{sec:MPC}. Theorem~\ref{thm:OptimalityConditions} is remarkable since it states that the classical KKT conditions are both necessary and sufficient local optimality conditions for the optimal control MPCC~\eqref{eq:LCCFTOC}.   This is in contrast to general MPCCs, for which the KKT conditions are sufficient for optimality, but typically not necessary~\cite[Ex.\ 3]{Scheel2000}.
We will present the KKT conditions from~\cite[Sec. 5.5.3]{Boyd2005},~\cite[Sec. 5.1]{Bertsekas2003} for~\eqref{eq:LCCFTOC} in a concise manner in Section~\ref{sec:MPC}.

\section{Structure of the Linear Complementarity Solutions}
\label{sec:LCP}

For the proof of Theorem~\ref{thm:OptimalityConditions} we will need a number of results that characterize the solution set of the complementarity constraint~\eqref{eq:LCMPCComplementarity} under the assumptions made in the previous section.

From Assumption~\ref{ass:DecomposablePSDRankOne} it is clear that we can focus on an individual complementarity sub-problem~\eqref{eq:SubComplementarityProblem}. Furthermore, for the structural results we derive below the actual values of the state $x$ and control input $u$ are of no consequence. Hence, to simplify the notation throughout this section we will consider the linear complementarity problem (LCP)
\begin{equation} \label{eq:LCP}
0 \leq M w + q \enspace \perp \enspace w \geq 0,
\end{equation}
where $w \in \R^{n}$, $M \in \R^{n \times n}$, and $q \in \R^{n}$. Note that we drop the sub-problem index $i$ from~\eqref{eq:SubComplementarityProblem} and assume $x$ and $u$ fixed to obtain
\[
q := C_i x + D_i u + e_i.
\]
Moreover by Assumptions~\ref{ass:DecomposablePSDRankOne} and~\ref{ass:NontrivialSolutions}
\begin{equation*}
\exists j \in \Nn{n} \colon \enspace q^{[j]} < 0 \quad \text{and} \quad M = m m^\top \succeq 0 %
\end{equation*}
for some $m \in \R^n$. %

Following the conventions of~\cite{Cottle1992}, we call the solution set of the LCP~\eqref{eq:LCP} for a fixed~$M$ and~$q$
\ifTwoColumn
\begin{align*}
\sol{q}{M} := \{ w \in \R^n \mid &\enspace  w \geq 0, \enspace Mw + q \geq 0, \\
	&\enspace w^\top \left(Mw + q \right) = 0 \}.
\end{align*}
\else
\[
\sol{q}{M} := \left\{ w \in \R^n \bigmid w \geq 0, \enspace Mw + q \geq 0, \enspace w^\top \left(Mw + q \right) = 0 \right\}.
\]
\fi
From $M = M^\top \succeq 0$ it follows that $\sol{q}{M}$ is a convex polytope:

\begin{lemma}[LCP solution set] \label{lem:PolyhedralSolutionSet}
The solution set for the complementarity problem~\eqref{eq:LCP} is the convex polytope
\[
\sol{q}{M} = \{ w \in \R^n_+ \mid  q^\top (w - \bar w) = 0, \; m^\top (w - \bar w) = 0 \}
\]
where $\bar w$ is any solution of~\eqref{eq:LCP}. %
\end{lemma}
\begin{proof}
The fact that $\sol{q}{M}$ is a convex polytope follows from $M \succeq 0$ and~\cite[Thm.\ 3.1.7]{Cottle1992}. We can simplify the notation from \cite{Cottle1992} by noting that $\nullspace(M) = \nullspace\left(m^\top\right)$.
\end{proof}
An immediate consequence of Lemma~\ref{lem:PolyhedralSolutionSet} is that the left-hand term $(Mw + q)$ in~\eqref{eq:LCP} has the same value for every $w \in \sol{q}{M}$~\cite[Thm.\ 3.4.4]{Cottle1992}. %

For a fixed $w \in \sol{q}{M}$ we introduce the  index sets
\begin{equation}
\begin{aligned}
\alpha(w) &:= \{i \in \Nn{n} \mid M^{[i,:]} w + q^{[i]} = 0, \enspace w^{[i]} > 0 \},\\
\beta(w)  &:= \{i \in \Nn{n} \mid M^{[i,:]} w + q^{[i]} = 0, \enspace w^{[i]} = 0 \},\\
\gamma(w) &:= \{i \in \Nn{n} \mid M^{[i,:]} w + q^{[i]} > 0, \enspace w^{[i]} = 0 \}, \\
\end{aligned} \label{eq:LCPComplementarityIndexSets}
\end{equation}
which partition $\Nn{n}$. Complementarity constraints for which $i \in \beta(w)$ are called \emph{biactive}, and the set $\beta(w)$ is the \emph{biactive set}. In case $\beta(w)$ is empty, $w$ is called \emph{nondegenerate}~\cite{Cottle1992}.

\begin{lemma}[nondegenerate tangent vectors] \label{lem:InteriorTangents}
Let $\bar w \in \sol{q}{M}$ solve the LCP~\eqref{eq:LCP} with
\[
\beta(\bar w) \neq \emptyset,
\]
i.e.\ assume there is at least one biactive complementarity constraint. For every $j \in \beta(\bar w)$ there exists another $w \in \sol{q}{M}$ such that
\begin{align*}
\alpha(w) = \alpha(\bar w) \cup \{ j \}
\enspace \text{and} \enspace \beta( w) = \beta(\bar w) \setminus \{ j \}.
\end{align*}
\end{lemma}
\begin{proof}
We will construct a vector $\delta \in \Re^n$ such that $w = (\bar w + \delta)$ satisfies the conditions of the lemma.   From Lemma~\ref{lem:PolyhedralSolutionSet}, any such $\delta$ must satisfy $w = (\bar w + \delta) \ge 0$ and $\delta \in \nullspace \left([m\,\,q]^{\smash\top}\right)$ in addition to the conditions on $\alpha(w)$ and $\beta(w)$.

Assumption~\ref{ass:NontrivialSolutions} ensures that $\alpha(\bar w)$ is always nonempty. Select an arbitrary $i \in \alpha(\bar w)$ and ${j \in \beta(\bar w)}$, and restrict $\delta$ to be zero aside from the elements $\delta^{[i]}$ and $\delta^{[j]}$.  We will choose the non-zero entries of $\delta$ such that
\begin{equation}\label{eq:LCPFeasibleNonZeros}
 \left\lvert \delta^{[i]} \right\rvert < \bar w^{[i]} \,\,\,  \text{and} \,\,\, \delta^{[j]} > 0,
\end{equation}
which is sufficient to ensure both the non-negativity of $w = \bar w + \delta$ and the required conditions on $\alpha(w)$ and $\beta(w)$.

Note that $q^{[i]} = -m^{[i]} (m^\top \bar w)$ (analogously for \smash{$q^{[j]}$}),
so that the nullspace condition on $\delta$ amounts to
\begin{equation}
\bm
m^{[i]} & m^{[j]} \\
-m^{[i]} \left( m^\top \bar w\right) &-m^{[j]} \left(m^\top \bar w\right)
\bmend%
\bm \delta^{[i]} \\ \delta^{[j]} \bmend = 0. \label{eq:LCPFeasibleDirectionSmallSystem}
\end{equation}

Recalling that Assumption~\ref{ass:DecomposablePSDRankOne} ensures that the vector $m$ is element-wise non-zero, it is then sufficient to choose
\[
\delta^{[i]} = -\frac{1}{2}\bar w^{[i]}\cdot \sign\left(\frac{m^{[i]}}{m^{[j]}}\right), \quad
\delta ^{[j]}= -\frac{m^{[i]}}{m^{[j]}}\delta^{{[i]}},
\]
which satisfies both \eqref{eq:LCPFeasibleNonZeros} and \eqref{eq:LCPFeasibleDirectionSmallSystem}.
\end{proof}

Lemma~\ref{lem:InteriorTangents} proves that in the presence of biactive complementarity constraints it is possible to find a tangential direction to $\sol{q}{M}$ such that exactly one index moves from $\beta$ to $\alpha$. This technical result is necessary to later prove our results on the tangent cone of the constraints in~\eqref{eq:LCCFTOC}.  Repeated application of this result provides the following corollary:

\begin{corollary}[strict complementarity] \label{cor:StrictComplementarity}
If $\sol{q}{M}$ is non-empty, then there exists a ${w \in \sol{q}{M}}$ with $\beta(w) = \emptyset$.
\end{corollary}

\section{Optimality Conditions for the Optimal Control Problem}
\label{sec:MPC}

In this section we will derive necessary and sufficient optimality conditions for the optimal control problem~\eqref{eq:LCCFTOC}.
The problem~\eqref{eq:LCCFTOC} is a specific case of the following general class of \emph{affine} MPCCs:
\begin{subequations}
\label{eq:MPCC}
\begin{align}
\min_v \quad &J(v)\quad   \label{eq:MPCCCost} \\
\text{s.t.} \quad &F_{in} v + f_{in} \leq 0, \quad F_{eq} v + f_{eq} = 0, \label{eq:MPCCLinearCon} \\
&\quad Gv + g \geq 0, \quad Hv + h \geq 0, \label{eq:MPCCPosCon} \\
&\quad\enspace(Gv + g)^\top (Hv + h) = 0 \label{eq:MPCCComplementarity}
\end{align}
\end{subequations}
Here, $v \in \R^{n_v}$ is the decision variable in~\eqref{eq:MPCC}, $f_{in} \in \R^p$, $f_{eq} \in \R^q$, $g,h \in \R^m$, and all other quantities have compatible dimension. For the optimal control problem~\eqref{eq:LCCFTOC} with $v := (\*u, \*x, \*w)$, we have
\begin{subequations}
\label{eq:ComplementarityMatricesLC}
\begin{align}
G &= \left( I_N \otimes E_u \enspace I_N \otimes E_x \enspace I_N \otimes E_w \right), &
g &=  \one_N \otimes e, \label{eq:ComplementarityMatricesLCG} \\
H &= \begin{pmatrix} 0 & 0 & I_{N n_w} \end{pmatrix}, &
h &= 0.\label{eq:ComplementarityMatricesLCH}
\end{align}
\end{subequations}
The other matrices can be constructed analogously. In the sequel we will require a number of concepts from the MPCC literature, and will make reference to the general affine MPCC~\eqref{eq:MPCC} to this end.
The optimization problem~\eqref{eq:MPCC} is nonconvex due to the complementarity constraint~\eqref{eq:MPCCComplementarity}. Note that there are many slightly different but equivalent formulations of~\eqref{eq:MPCC}, in particular regarding the complementarity constraint.
We will use the formulation~\eqref{eq:MPCC} throughout the paper, which is without loss of generality (see~\cite{Luo1996}).

\subsection{Preliminaries: stationarity conditions for MPCCs}
\label{subsec:MPCCStationarity}

We introduce a number of index sets relating to the active (complementarity) constraints for a given feasible $v^*$:
\ifTwoColumn
\begin{align}
\alpha &:= \{i \in \Nn{m} \mid G^{[i,:]} v^* + g^{[i]} = 0, \enspace H^{[i,:]} v^* + h^{[i]} > 0 \}, \nonumber\\
\beta  &:= \{i \in \Nn{m} \mid G^{[i,:]} v^* + g^{[i]} = 0, \enspace H^{[i,:]} v^* + h^{[i]} = 0 \},\nonumber\\
\gamma &:= \{i \in \Nn{m} \mid G^{[i,:]} v^* + g^{[i]} > 0, \enspace H^{[i,:]} v^* + h^{[i]} = 0 \}, \nonumber\\
\mathcal{I}_{in} &:= \{i \in \Nn{p} \mid F_{in}^{[i,:]} v^* + f_{in}^{[i]} = 0 \}.
\label{eq:ActiveComplementarityConstraints}
\end{align}
\else
\begin{equation}
\begin{aligned}
\alpha &:= \{i \in \Nn{m} \mid G^{[i,:]} v^* + g^{[i]} = 0, \enspace H^{[i,:]} v^* + h^{[i]} > 0 \},\\
\beta  &:= \{i \in \Nn{m} \mid G^{[i,:]} v^* + g^{[i]} = 0, \enspace H^{[i,:]} v^* + h^{[i]} = 0 \},\\
\gamma &:= \{i \in \Nn{m} \mid G^{[i,:]} v^* + g^{[i]} > 0, \enspace H^{[i,:]} v^* + h^{[i]} = 0 \}, \\
\mathcal{I}_{in} &:= \{i \in \Nn{p} \mid F_{in}^{[i,:]} v^* + f_{in}^{[i]} = 0 \}.
\end{aligned} \label{eq:ActiveComplementarityConstraints}
\end{equation}
\fi
Due to the analogy between these definitions and~\eqref{eq:LCPComplementarityIndexSets} the same letters are generally used in the literature to identify the singularly active and biactive complementarity constraints.
Note that we omit the dependence of these sets on the feasible $v^*$ in the interest of a simpler notation whenever it is unambiguous. %
When $\beta = \emptyset$ we say \emph{strict complementarity} holds~\cite{Luo1996}.

We define the Lagrangian function for~\eqref{eq:MPCC} according to~\cite{Boyd2005,Bertsekas2003} as
\ifTwoColumn
\begin{align*}
&\mathcal{L}(v, \hat \eta, \hat \mu, \hat \nu_G, \hat \nu_H, \hat \xi) := J(v) + \hat \eta^\top \left(F_{in} v + f_{in} \right) \\
&\quad + \hat \mu^\top \left(F_{eq} v + f_{eq} \right) - \hat \nu_G^\top \left(G v + g \right) - \hat \nu_H^\top \left( H v + h \right) \\
&\quad + \hat \xi \left(G v + g \right)^\top \left( H v + h \right),
\end{align*}
\else
\begin{align*}
\mathcal{L}(v, \hat \eta, \hat \mu, \hat \nu_G, \hat \nu_H, \hat \xi) &:= J(v) + \hat \eta^\top \left(F_{in} v + f_{in} \right) + \hat \mu^\top \left(F_{eq} v + f_{eq} \right) \\
&\qquad- \hat \nu_G^\top \left(G v + g \right) - \hat \nu_H^\top \left( H v + h \right) + \hat \xi \left(G v + g \right)^\top \left( H v + h \right),
\end{align*}
\fi
where we introduced KKT multiplier variables $\hat \eta \in \R^{p}$, $\hat \mu \in \R^{q}$, $\hat \nu_G \in \R^{m}$, $\hat \nu_H  \in \R^{m}$, and $\hat \xi \in \R$ for the constraints in~\eqref{eq:MPCC}.
The classical KKT conditions for~\eqref{eq:MPCC} are then given as follows~\cite[Sec. 5.5.3]{Boyd2005}, \cite[Sec. 5.1]{Bertsekas2003}:
\begin{gather}
\nabla_v \mathcal{L}(v, \hat \eta, \hat \mu, \hat \nu_G, \hat \nu_H, \hat \xi) = 0, \nonumber \\
\begin{aligned}
\hat\eta^{[i]} &\geq 0 \quad \text{for} \quad i \in \mathcal{I}_{in}, & \hat\eta^{[i]} &= 0 \quad \text{for} \quad i \not\in \mathcal{I}_{in}, \\
\hat\nu_G^{[i]} &= 0 \quad \text{for} \quad i \in \gamma, & \hat\nu_H^{[i]} &= 0 \quad \text{for} \quad i \in \alpha,
 \\
\hat\nu_G^{[i]} &\geq 0 \quad \text{for} \quad i \in \alpha \cup \beta, & \hat\nu_H^{[i]} &\geq 0 \quad \text{for} \quad i \in \beta \cup \gamma.
\end{aligned} \label{eq:KKTConditions}
\end{gather}
A set of primal variables $v^* \in \R^{n_v}$ feasible in~\eqref{eq:MPCC} and KKT multiplier variables (also called \emph{Lagrange multipliers}) that together satisfy the KKT conditions~\eqref{eq:KKTConditions} are also called a \emph{primal-dual solution pair} for~\eqref{eq:MPCC}.
As discussed in the introduction the KKT conditions~\eqref{eq:KKTConditions} may not be necessary optimality conditions for~\eqref{eq:MPCC} due to failure of MFCQ~\cite{Chen1995}. A number of alternative stationarity conditions to the classical KKT conditions have therefore been introduced. We will limit the discussion to the concepts necessary for our control context; for a comprehensive discussion of weaker stationarity conditions that are applicable for more general MPCCs see~\cite{Ye2005,Flegel2005}.

The strongest optimality conditions that can generally be expected to hold at an optimal point~\cite{Hoheisel2011} are the so-called \emph{M(ordukhovich)}-stationary conditions:
\begin{definition}[M-stationarity] \label{def:MStationarity}
A feasible point $v^* \in \R^{n_v}$ of~\eqref{eq:MPCC} is called M-stationary if there exist $\eta \in \R^p$, $\mu \in \R^q$, $\nu_G \in \R^m$, and $\nu_H \in \R^m$ such that
\begin{gather}
\nabla J (v^*) + F_{in}^\top \eta + F_{eq}^\top \mu - G^\top \nu_G - H^\top \nu_H = 0,  \nonumber \\
\begin{aligned}
\eta^{[i]} &\geq 0 \quad \text{for} \quad i \in \mathcal{I}_{in}, & \eta^{[i]} &= 0 \quad \text{for} \quad i \not\in \mathcal{I}_{in}, \\
\nu_G^{[i]} &= 0 \quad \text{for} \quad i \in \gamma, & \nu_H^{[i]} &= 0 \quad \text{for} \quad i \in \alpha,
\end{aligned} \label{eq:MStationarity} \\
(\nu_G^{[i]} \geq 0 \wedge \nu_H^{[i]} \geq 0) \quad \vee \quad \nu_G^{[i]} \nu_H^{[i]} = 0 \quad \text{for} \quad i \in \beta. \nonumber
\end{gather}
\end{definition}

The quantities $\eta$, $\nu_G$, and $\nu_H$ are commonly referred to as \emph{Lagrange multipliers}, although they are \emph{distinct} from the classical Lagrange multipliers appearing in the KKT conditions~\eqref{eq:KKTConditions}. To avoid confusion we refer to the former here as the MPCC multipliers and use the term KKT multipliers for the latter.
The constraints on the signs of these MPCC multipliers are different, i.e.\ the multipliers for the singularly active complementarity constraints, $\nu_G^{[i]}$ for $i \in \alpha$ and $\nu_H^{[i]}$ for $i \in \gamma$, are allowed to be negative. The same holds for one of the multipliers corresponding to the biactive complementarity constraints, $\nu_G^{[i]}$ and $\nu_H^{[i]}$ for $i \in \beta$, as long as the other one is equal to zero.
While it can be shown that the M-stationarity conditions correspond to non-smooth KKT conditions  for an equivalent formulation of~\eqref{eq:MPCC} \cite{Ye2005}, it is possible that an M-stationary local minimum $v^*$ of~\eqref{eq:MPCC} does not admit a primal-dual solution to the classical KKT conditions~\eqref{eq:KKTConditions}~\cite[Ex.~3]{Scheel2000}.

The so-called \emph{strong stationarity} conditions are stronger than those of M-stationarity and have a close relation to the KKT conditions of~\eqref{eq:MPCC}:
\begin{definition}[S-stationarity]\label{def:StrongStationarity}
A feasible point $v^* \in \R^{n_v}$ of~\eqref{eq:MPCC} is strongly (or S-)stationary if there exist $\eta \in \R^p$, $\mu \in \R^q$, $\nu_G \in \R^m$, and $\nu_H \in \R^m$ such that
\begin{gather}
\nabla J (v^*) + F_{in}^\top \eta + F_{eq}^\top \mu - G^\top \nu_G - H^\top \nu_H = 0,  \nonumber \\
\begin{aligned}
\eta^{[i]} &\geq 0 \quad \text{for} \quad i \in \mathcal{I}_{in}, & \eta^{[i]} &= 0 \quad \text{for} \quad i \not\in \mathcal{I}_{in}, \\
\nu_G^{[i]} &= 0 \quad \text{for} \quad i \in \gamma, & \nu_H^{[i]} &= 0 \quad \text{for} \quad i \in \alpha,
\end{aligned} \label{eq:SStationarity} \\
\nu_G^{[i]} \geq 0 \wedge \nu_H^{[i]} \geq 0 \quad \text{for} \quad i \in \beta. \nonumber
\end{gather}
\end{definition}
The only difference between M- and S-stationarity lies in the conditions on the biactive multipliers. Hence, the two stationarity conditions (and in fact most other MPCC stationarity conditions) collapse into S-stationarity in the case of strict complementarity ($\beta = \emptyset$). %

A feasible point $v^*$ of~\eqref{eq:MPCC} is an S-stationary point if and only if there exist KKT multipliers satisfying the classical KKT conditions~\eqref{eq:KKTConditions} for~\eqref{eq:MPCC} at the same point~\cite[Prop.\ 4.2]{Flegel2005a}.  However, the classical KKT multipliers in~\eqref{eq:KKTConditions} are \emph{distinct} from the MPCC multipliers certifying S-stationarity in~\eqref{eq:SStationarity}, e.g.\ the MPCC multipliers characterizing an S-stationary point do not include a multiplier for the orthogonality constraint~\eqref{eq:MPCCComplementarity}.
It can be shown that MPCC multipliers $\nu_G$ and $\nu_H$ certifying strong stationarity of $v^*$ can be computed from KKT multipliers satisfying~\eqref{eq:KKTConditions} as~
\ifTwoColumn
\begin{subequations}
\label{eq:MultiplierConversion}
\begin{align}
\nu_G &= \hat \lambda_G - \hat\xi \left( H v^* + h \right), \\
\nu_H &= \hat \lambda_H - \hat\xi \left( G v^* + g \right).
\end{align}
\end{subequations}
\else
\begin{align}
\nu_G &= \hat \lambda_G - \hat \xi \left( H v^* + h \right), &
\nu_H &= \hat \lambda_H - \hat \xi \left( G v^* + g \right). \label{eq:MultiplierConversion}
\end{align}
\fi
While $\hat \lambda_G$ and $\hat \lambda_H$ must be nonnegative because they originate from the classical KKT conditions, individual components of $\nu_G$ and $\nu_H$ might be negative as indicated in Definition~\ref{def:StrongStationarity}. Derivation of KKT multipliers from given MPCC multipliers certifying S-stationarity is similarly possible~\cite{Flegel2005a}.

While M-stationarity is a necessary optimality condition for the affine MPCC~\eqref{eq:MPCC}~\cite{Ye2005}, additional conditions on the MPCC multipliers are required to make M-stationarity a sufficient condition for local optimality, e.g.~\cite[Thm.~2.3]{Ye2005}. S-stationarity, on the other hand, is a sufficient optimality condition for the affine MPCC~\eqref{eq:MPCC}, but is not a necessary optimality condition in general~\cite[Ex.~3]{Scheel2000}. We show, however, that in the particular problem~\eqref{eq:LCCFTOC} considered here the S-stationarity conditions are both necessary and sufficient for optimality as stated in Theorem~\ref{thm:OptimalityConditions}.

Our proof of Theorem~\ref{thm:OptimalityConditions} is based on a regularity property of the \emph{linearized tangent cone} $\mathcal{T}^{lin}(v^*)$ to the constraints of~\eqref{eq:MPCC} at a feasible $v^*$ \cite{Flegel2005b} defined as
\[
\mathcal{T}^{lin}(v^*) := \left\{ d \in \R^{n_v} \bigmid \begin{aligned}
\forall i &\in \mathcal{I}_{in} \colon & F_{in}^{[i,:]} d &\leq 0 \\
&& F_{eq} d &= 0 \\
\forall i &\in \alpha \colon & G^{[i,:]} d &= 0 \\
\forall i &\in \gamma \colon & H^{[i,:]} d &= 0 \\
\forall i &\in \beta \colon & G^{[i,:]} d &\geq 0 \\
\forall i &\in \beta \colon & H^{[i,:]} d &\geq 0 \\
\end{aligned} \right\}.
\]
and of the related \emph{MPEC-linearized tangent cone} $\mathcal{T}^{lin}_{MPEC}(v^*)$~\cite{Ye2005,Flegel2005a}:
\ifTwoColumn
\begin{equation}
\begin{gathered}
\mathcal{T}^{lin}_{MPEC}(v^*) :=  \mathcal{T}^{lin}(v^*) \enspace \cap \\
 \left\{ d \in \R^{n_v} \bigmid \forall i \in \beta \colon (G^{[i,:]} d) (H^{[i,:]} d) = 0 \right\}
\end{gathered} \label{eq:MPECLinTangentCone}
\end{equation}
\else
\begin{equation}
\mathcal{T}^{lin}_{MPEC}(v^*) := \mathcal{T}^{lin}(v^*) \cap \left\{ d \in \R^{n_v} \bigmid \forall i \in \beta \colon (G^{[i,:]} d) (H^{[i,:]} d) = 0 \right\} \label{eq:MPECLinTangentCone}
\end{equation}
\fi
The \emph{tangent cone} $\mathcal{T}(v^*)$ to the feasible set $\mathcal{V}$ of~\eqref{eq:MPCC} is a closed cone~\cite{Rockafellar1998,Pang1999} and defined as
\ifTwoColumn
\begin{align*}
\mathcal{T}(v^*) := \bigg\{ d \in \R^{n_v} \mid &~\exists t_k \in \R_+, v_k \in \mathcal{V} \colon  \\
&\lim_{k \to \infty } t_k = 0 \enspace \wedge \enspace \lim_{k \to \infty} \frac{v_k - v^*}{t_k} = d \bigg\}.
\end{align*}
\else
\[
\mathcal{T}(v^*) := \left\{ d \in \R^{n_v} \bigmid \exists t_k \in \R_+, v_k \in \mathcal{V} \colon \lim_{k \to \infty } t_k = 0 \enspace \wedge \enspace \lim_{k \to \infty} \frac{v_k - v^*}{t_k} = d \right\}.
\]
\fi
It holds that $\mathcal{T}(v^*) \subseteq \mathcal{T}^{lin}_{MPEC}(v^*) \subseteq \mathcal{T}^{lin}(v^*)$~\cite{Flegel2005} and, hence, $\left(\mathcal{T}^{lin}(v^*) \right)^\circ \subseteq \left(\mathcal{T}^{lin}_{MPEC}(v^*)\right)^\circ \subseteq \left( \mathcal{T}(v^*) \right)^\circ$, where $C^\circ$ is the polar cone of $C$~\cite[Sec.\ 6.E]{Rockafellar1998}. Also note that $\mathcal{T}^{lin}(v^*)$ is a convex polyhedral cone while $\mathcal{T}(v^*)$ and $\mathcal{T}^{lin}_{MPEC}(v^*)$ are generally nonconvex.

\subsection{Necessary and sufficient optimality conditions for the optimal control problem}

We can now prove an important geometric property of the feasible set of the optimal control problem~\eqref{eq:LCCFTOC}. Specifically, that it satisfies the \emph{intersection property}~\cite{Flegel2007}:

\begin{lemma}[intersection property] \label{lem:IntersectionProperty}
The feasible set of~\eqref{eq:LCCFTOC} for an LC model satisfying Assumptions~\ref{ass:WellPosedness},~\ref{ass:DecomposablePSDRankOne}, and~\ref{ass:NontrivialSolutions} satisfies the intersection property at every feasible $\*v := (\*u,\*x,\*w)$, i.e.\
\begin{equation}
\left(\mathcal{T}^{lin}_{MPEC}(\*v) \right)^\circ = \left( \mathcal{T}^{lin}(\*v) \right)^\circ. \label{eq:IntersectionProperty}
\end{equation}
\end{lemma}
\begin{proof}
If $\beta = \emptyset$ then the result is obvious, so we assume $\beta \neq \emptyset$ throughout and will use Lemma 1 from~\cite{Pang1999} to prove that equality holds in~\eqref{eq:IntersectionProperty}. To this end we will for every $j \in \beta$ find a vector $d \in \R^{n_v}$ such that
\begin{subequations}
\label{eq:RecessionVectorConstraints}
\ifTwoColumn
\begin{gather}
H^{[j,:]} d > 0, \enspace
F_{eq}d = 0,  \enspace F_{in}^{[i,:]} d = 0 \enspace \forall i \in \mathcal{I}_{in}, \label{eq:RecessionVectorConstraintsLinear} \\
G^{[i,:]} d = 0 \ \forall i \in \alpha \cup \beta,  \ H^{[i,:]} d = 0 \ \forall i \in \gamma \cup \beta \setminus \{j\}, \label{eq:RecessionVectorConstraintsLCP}
\end{gather}
\else
\begin{align}
F_{eq}d &= 0, & F_{in}^{[i,:]} d &= 0 \enspace \forall i \in \mathcal{I}_{in} \label{eq:RecessionVectorConstraintsLinear} \\
G^{[i,:]} d &= 0 \enspace \forall i \in \alpha \cup \beta, & H^{[i,:]} d &= 0 \enspace \forall i \in \gamma \cup \beta \setminus \{j\}, & H^{[j,:]} d &> 0, \label{eq:RecessionVectorConstraintsLCP}
\end{align}
\fi
\end{subequations}
where we used the more compact notation of~\eqref{eq:MPCC}. It can be seen that the $d$ we construct must be a local recession vector to the feasible set of~\eqref{eq:LCCFTOC} such that the biactive complementarity constraint $j$ becomes singularly active, i.e.\ for a sufficiently small $\epsilon$ we have
\[
\alpha(v + \epsilon d) = \alpha(v) \cup \{j \} \text{ and } \beta(v + \epsilon d) = \beta(v) \setminus \{j \}.
\]
The similarity with Lemma~\ref{lem:InteriorTangents} is not coincidental, we will in fact use it in the proof.

Let the recession vector $d := (\mathbf{d_u}, \mathbf{d_x}, \mathbf{d_w})$ be such that $\mathbf{d_u} = 0$ and $\mathbf{d_x} = 0$. Comparing~\eqref{eq:LCCFTOC} with~\eqref{eq:RecessionVectorConstraints} it is easy to see that~\eqref{eq:RecessionVectorConstraints} in that case decouples over the horizon. The complementarity constraint $j$ we are considering is part of~\eqref{eq:LCMPCComplementarity} for a particular timestep $ k$ and we can set all components of $\mathbf{d_w} := \left(d_{w_0}, \dots, d_{w_{N-1}}\right)$ not corresponding to $ k$ equal to zero.

We have now reduced the conditions~\eqref{eq:RecessionVectorConstraints} to finding a special recession direction for a single LCP~\eqref{eq:LCMPCComplementarity}.
Due to Assumption~\ref{ass:DecomposablePSDRankOne} this problem decouples even further into the linear complementarity sub-problems~\eqref{eq:SubComplementarityProblem} in the variables $w_{k,i}$. Hence, by Assumption 3 and Lemma~\ref{lem:InteriorTangents} there exists a direction $d$ satisfying~\eqref{eq:RecessionVectorConstraintsLCP}.

From the proof of Lemma~\ref{lem:InteriorTangents} we can see that this $d_{w_k} \in \nullspace(E_w)$ and, hence, by Lemma~\ref{lem:NecessaryConditionWellPosedness} also $d_{w_k} \in \nullspace(B_w)$. Any additional linear inequality constraints present in the optimal control problem~\eqref{eq:LCCFTOC} would  constrain only states and control inputs and the corresponding components of the recession vector $d$ are assumed zero.

The  argument above proves that  we can construct for any $j \in \beta$ a recession vector satisfying~\eqref{eq:RecessionVectorConstraints}. This is equivalent to condition $\mathbf{(A_H)}$ in Lemma~1 of~\cite{Pang1999} with $\beta_{GH_1} = \emptyset$ and $\beta_{GH_2} = \beta$. By the same Lemma and Theorem~1 in the same reference it follows that the intersection property holds.
\end{proof}

With the intersection property established for~\eqref{eq:LCCFTOC} we are now in a position to prove Theorem~\ref{thm:OptimalityConditions}:

\textbf{Proof of Theorem~\ref{thm:OptimalityConditions}:} Consider the tangent cone $\mathcal{T}\left(\*v^*\right)$ to the constraints in~\eqref{eq:LCCFTOC} at the feasible $\*v^*$ from the theorem (for details see~\cite{Flegel2007,Rockafellar1998}). We have, e.g.\ from~\cite{Flegel2005},
\begin{equation}
\mathcal{T}(\*v^*) \subseteq \mathcal{T}^{lin}_{MPEC}(\*v^*) \subseteq \mathcal{T}^{lin}(\*v^*).\label{eq:TangentConeInclusions}
\end{equation}
Since~\eqref{eq:LCCFTOC} is an affine MPCC the MPEC-Abadie constraint qualification holds at every feasible $\*v$~\cite[Thm.~3.2]{Flegel2005b} and in particular at $\*v^*$, which means $\mathcal{T}(\*v^*) = \mathcal{T}^{lin}_{MPEC}(\*v^*)$~\cite[Def.\ 3.1]{Flegel2005b}. Substituting this and considering the polar cones in~\eqref{eq:TangentConeInclusions} we obtain
\[
\left(\mathcal{T}(\*v^*)\right)^\circ = \left(\mathcal{T}^{lin}_{MPEC}(\*v^*)\right)^\circ \supseteq \left(\mathcal{T}^{lin}(\*v^*)\right)^\circ.
\]
From Lemma~\ref{lem:IntersectionProperty} we know that the intersection property holds for~\eqref{eq:LCCFTOC} and, hence, the inclusion is actually an equality and we finally have
\[
\left(\mathcal{T}(\*v^*)\right)^\circ = \left(\mathcal{T}^{lin}(\*v^*)\right)^\circ.
\]
This is the so-called \emph{Guignard constraint qualification}. It follows from \cite[Thm.~16]{Flegel2007} that S-stationarity, in the sense of Definition~\ref{def:StrongStationarity}, is a necessary optimality condition for~\eqref{eq:LCCFTOC}. We also know from~\cite[Thm.~2.3]{Ye2005} that S-stationarity is a sufficient optimality condition for~\eqref{eq:LCCFTOC}. S-stationarity is in fact equivalent to the classical KKT-conditions for~\eqref{eq:LCCFTOC}~\cite[Prop.\ 4.2]{Flegel2005a}. This completes the proof.
\hfill $\square$

Note that the conclusion of Theorem~\ref{thm:OptimalityConditions} also holds with additional linear constraints on the states and inputs, because the recession direction $d$ constructed in the proof of Lemma~\ref{lem:IntersectionProperty} has no components in the $x$ and $u$ directions, i.e.\ the intersection property will still hold.

\subsection{Stronger conditions for global and isolated optima}

We now present a number of stronger sufficient conditions that can be used to verify whether a given locally optimal solution $\*v^* := (\*u^*,\*x^*,\*w^*)$ is, in fact, \emph{globally optimal} or an \emph{isolated minimizer}. The latter means that there is no other locally optimal point within a neighborhood around $\*v^*$.
To present sufficient conditions for an S-stationary solution $\*v^*$ of~\eqref{eq:LCCFTOC} to be globally optimal, we introduce MPCC multipliers $\mu_k \in \R^{n_x}$ for~\eqref{eq:LCMPCDynamicAssignment} and $\nu_{k}, \lambda_k \in \R^{n_w}$ for~\eqref{eq:LCMPCComplementarity} to present the S-stationarity conditions for $\*v^*$:
\begin{subequations}
\label{eq:SStationarityConditions}
\ifTwoColumn
\begin{align}
\frac{\partial \ell_N}{\partial x_N}( x^*_N) + \mu_{N-1} &= 0 \label{eq:SStatCondStatXN} \\
\forall k = 1,\dots,N-1 \colon  \hspace{3cm} & \nonumber \\
\frac{\partial \ell_k}{\partial x_k}( x^*_k, u^*_k) + \mu_{k-1} - A^\top \mu_k - E_x^\top \nu_k &= 0 \label{eq:SStatCondStatXk} \\
\forall k = 0,\dots,N-1 \colon \hspace{3cm} & \nonumber \\
\frac{\partial \ell_k}{\partial u_k}( x^*_k, u^*_k) - B_u^\top \mu_k - E_u^\top \nu_k  &= 0  &&\label{eq:SStatCondStatUk} \\
 - B_w^\top \mu_k - E_w^\top \nu_k - \lambda_k &= 0 \label{eq:SStatCondStatWk}
\end{align}
\else
\begin{align}
&&\frac{\partial \ell_N}{\partial x_N}( x_N^*) + \mu_{N-1} &= 0 \label{eq:SStatCondStatXN} \\
\forall k = 1,\dots,N-1 &\colon &%
\frac{\partial \ell_k}{\partial x_k}( x^*_k, u^*_k) + \mu_{k-1} - A^\top \mu_k - E_x^\top \nu_k &= 0 \label{eq:SStatCondStatXk} \\
\forall k = 0,\dots,N-1 &\colon &
\frac{\partial \ell_k}{\partial u_k}( x_k^*,  u_k^*) - B_u^\top \mu_k - E_u^\top \nu_k  &= 0  &&\label{eq:SStatCondStatUk} \\
\forall k = 0,\dots,N-1 &\colon &
 - B_w^\top \mu_k - E_w^\top \nu_k - \lambda_k &= 0 \label{eq:SStatCondStatWk}
\end{align}
\fi
We omit the (inconsequential) conditions related to the derivatives with respect to $x_0$.
Since the complementarity constraints in~\eqref{eq:LCMPCComplementarity} naturally decouple over the prediction horizon for the given S-stationary $\*v^*$, we can consider index sets $\alpha_k$, $\beta_k$, and $\gamma_k$ analogously defined as in~\eqref{eq:ActiveComplementarityConstraints} for each stagewise LCP. With that, we have the following conditions for the MPCC multipliers corresponding to the complementarity constraints~\eqref{eq:LCMPCComplementarity}:
\begin{gather}
\forall i \in \alpha_k \colon \quad \lambda_k^{[i]} = 0, \qquad
\forall i \in \gamma_k \colon \quad \nu_k^{[i]} = 0, \label{eq:SStatCondComplAlphaGamma} \\
\forall i \in \beta_k \colon \quad \nu_{k}^{[i]} \geq 0, \enspace \lambda_{k}^{[l]} \geq 0 \label{eq:SStatCondComplBiactive}
\end{gather}
\end{subequations}
We can now state a sufficient condition for an S-stationary solution of~\eqref{eq:LCCFTOC} to be globally optimal.

\begin{theorem}[sufficient global optimality condition] \label{thm:SufficientFirstOrderCondition}
Let $\*v^*=( \*u^*,\*x^*, \*w^*)$ be an S-stationary point for~\eqref{eq:LCCFTOC}, i.e.\ there exist MPCC multipliers $\mu_k \in \R^{n_x}$, $\nu_{k} \in \R^{n_w}$, and $\lambda_{k} \in \R^{n_w}$ such that~\eqref{eq:SStationarityConditions} holds. If additionally for all $ k \in \Nn{N-1}$ %
\begin{align}
\nu_k \geq 0 \text{ and } \lambda_k \geq 0, \label{eq:SufficientOptimalityConditions}
\end{align}
then $\*v^*$ is a globally optimal solution to~\eqref{eq:LCCFTOC}.
\end{theorem}
\begin{proof}
Since all constraint functions in~\eqref{eq:LCCFTOC} are affine and we assumed convex stage- and terminal cost functions $\ell_k$ the result follows from~\cite[Thm.~2.3]{Ye2005}.
\end{proof}

In addition to the first-order optimality conditions from Theorems~\ref{thm:OptimalityConditions} and~\ref{thm:SufficientFirstOrderCondition} we can use second-order sufficient conditions to identify an isolated minimizer of~\eqref{eq:LCCFTOC}.
These pose conditions on the \emph{critical directions} at a stationary point $v^*$ of~\eqref{eq:MPCC} to ensure that no feasible non-ascent direction exists.
To this end, we introduce the so-called \emph{MPEC critical cone} based on the MPEC-linearized tangent cone $\mathcal{T}_{MPEC}^{lin}(v)$ from~\eqref{eq:MPECLinTangentCone}:
\[
\mathcal{C}_{MPEC}(v^*) := \mathcal{T}_{MPEC}^{lin}(v^*) \cap \left\{ d \in \R^{n_v} \bigmid  \nabla J(v^*)^\top d \leq 0 \right\}
\]
With this we introduce the following second-order sufficient condition:

\begin{definition}[strong second-order sufficient condition] \label{def:MSSOSC}
Given an M-stationary point $v^* \in \R^{n_v}$ for~\eqref{eq:MPCC} we say that the \emph{M-multiplier strong second-order sufficient condition} (M-SSOSC) holds at $v^*$ if and only if
\[
\forall d \in \mathcal{C}_{MPEC}(v^*) \setminus \{ 0 \} \colon \quad d^\top \nabla^2 J(v^*) d > 0.
\]
\end{definition}
Note that for~\eqref{eq:MPCC}, M-SSOSC does not depend on the values of any MPCC multipliers but instead only on the critical directions. %

For a comprehensive discussion of this and other second-order conditions for general MPCCs we refer to~\cite{Guo2013}. Note that we have simplified substantially the definitions from~\cite{Guo2013} since we are only considering the affine MPCC~\eqref{eq:MPCC}. We can use the M-SSOSC from Definition~\ref{def:MSSOSC} to identify an isolated minimizer $\*v^*$ of~\eqref{eq:LCCFTOC}, i.e.\ the only local minimizer within a neighborhood around $\*v^*$.

\begin{theorem}[isolated minimizer with unique complementarity variables] \label{thm:IsolatedMinimizer}
Let $\*v^* =(\*u^*,\*x^*,\*w^*)$ be an S-stationary point for~\eqref{eq:LCCFTOC} and let M-SSOSC hold at $\*v^*$.
Then $\*v^*$ is an isolated local minimizer, i.e.\ there exists an $\epsilon > 0$ such that
\ifTwoColumn
\begin{gather*}
\lVert \*v - \*v^* \rVert \leq \epsilon \quad \Rightarrow \\
 \ell_N(x^*_N) + \sum_{k=0}^{N-1} \ell_k(x^*_k, u^*_k) < \ell_N(x_N) + \sum_{k=0}^{N-1} \ell_k(x_k, u_k).
\end{gather*}
\else
\[
\lVert \*v - \*v^* \rVert \leq \epsilon \quad \Rightarrow \quad \ell_N(x^*_N) + \sum_{k=0}^{N-1} \ell_k(x^*_k, u^*_k) < \ell_N(x_N) + \sum_{k=0}^{N-1} \ell_k(x_k, u_k).
\]
\fi
Furthermore, $\*w^*$ are the only complementarity variables solving the LCP~\eqref{eq:LCMPCComplementarity} for the fixed $\*u^*$ and $\*x^*$.
\end{theorem}
\begin{proof}
The fact that $\*v^*$ is an isolated minimizer follows from~\cite[Cor.\ 4.3]{Guo2013}.

Assume there exists another $\mathbf{\bar w}$ such that $\mathbf{\bar v} := (\*u^*, \*x^*, \mathbf{\bar w})$ is feasible for~\eqref{eq:LCCFTOC}. Since the LCP~\eqref{eq:LCMPCComplementarity} decouples over the horizon by Lemma~\ref{lem:PolyhedralSolutionSet} the set of feasible $\*w$ is a convex polytope, hence $\mathbf{\bar w}$ can be arbitrarily close to $\*w^*$. Since $\mathbf{\bar w}$ does not enter the cost-function~\eqref{eq:LCMPCCost} we have $J(\mathbf{\bar v}) = J(\*v^*)$ and reached a contradiction. %
\end{proof}

\section{Optimality Conditions for Optimal Control Inputs}
\label{sec:OptimalTrajctories}

The proof of Theorem~\ref{thm:IsolatedMinimizer} hints at a possible challenge when solving the optimal control problem~\eqref{eq:LCCFTOC} that, to the authors' knowledge, has not yet been considered in the literature: we are interested in finding optimal control input trajectories $\*u^*$ and the corresponding state trajectory~$\*x^*$, but the optimal control problem also includes the optimal complementarity variables $\*w^*$. Fundamentally, the problem we want to solve is the hybrid optimal control problem
\begin{subequations}
\label{eq:OCP}
\begin{align}
\min_{\*u,\*x} \quad  &{\ell}_N(x_N) + \sum_{k=0}^{N-1} {\ell}_k(x_k,u_k) \label{eq:OCPCost} \\
\text{s.t.} \quad &\forall k = 0,\dots,{N-1} \colon \quad \label{eq:OCPDynamics} %
x_{k+1} = f(x_k,u_k),
\end{align}
\end{subequations}
where $f \colon \R^{n_x} \times \R^{n_u} \to \R^{n_x}$ describes the LC dynamics~\eqref{eq:LCModel}. The problem~\eqref{eq:LCCFTOC} is an instance of~\eqref{eq:OCP} that  explicitly considers complementarity variables $w$ to facilitate theoretical derivations and numerical computations.

It is conceivable that a local optimum $\*v^*= (\*u^*, \*x^*, \*w^*)$ of~\eqref{eq:LCCFTOC}, which is a member of the feasible set
\[
\mathcal{Q} := \left\{ \*v= (\*u, \*x, \*w) \bigmid \*v \text{ satisfies the constraints in~\eqref{eq:LCCFTOC}} \right\},
\]
has no corresponding optimal solution $(\*u^*, \*x^*)$ to~\eqref{eq:OCP} in the set
\[
\mathcal{U} := \left\{ (\*u, \*x) \bigmid \exists \*w \colon (\*u, \*x, \*w) \in \mathcal{Q} \right\},
\]
which is the projection of $\mathcal{Q}$ onto the state-input space and the feasible set of~\eqref{eq:OCP}.

We will next present an example to illustrate this point. To this end we introduce the set
\[
\mathcal{M}(\*u, \*x) := \left\{ \*w \bigmid (\*u, \*x, \*w) \in \mathcal{Q} \right\}
\]
of complementarity variables that are consistent in~\eqref{eq:LCCFTOC} with a given $\*u$ and $\*x$.

\subsection{Example: non-optimal input trajectories can be S-stationary}
\label{subsec:Example}

Consider the LC model %
\begin{subequations}
\label{eq:Ex1Dynamics}
\begin{alignat}{3}
x^+ &=  w^{[1]} + w^{[2]} -2,  \\
0 &\leq w^{[1]} + w^{[2]} + x + u &\quad & \perp \quad & w^{[1]}  &\geq 0, \\
0 &\leq w^{[1]} + w^{[2]} - 1  &&\perp & w^{[2]} &\geq 0,
\end{alignat}
\end{subequations}
which is equivalent to the continuous PWA model
\[
x^+ = \max\{-(x+u+2),-1\}.
\]
It is easy to verify that~\eqref{eq:Ex1Dynamics} satisfies Assumptions~\ref{ass:WellPosedness}--\ref{ass:NontrivialSolutions}, the details are omitted here for brevity.

Consider the optimal control problem~\eqref{eq:LCCFTOC} with initial state $x_0 = 0$, prediction horizon $N = 1$, stage cost $\ell_0(x_0,u_0) = \frac{1}{2} u_0^2$, and terminal cost $\ell_1(x_1) = \frac{1}{2} x_1^2$. For a control input $u_0 = -1$ we obtain $x_1 = -1$ and an infinite number of admissible complementarity variables with
\begin{align}
\mathcal{M}(-1, -1) =\left\{ w_0 \in \R_+^2 \bigmid w_0^{[1]} + w_0^{[2]} = 1 \right\}.  \label{eq:Example1InnerMultipliers}
\end{align}
Solving the S-stationarity conditions~\eqref{eq:SStationarity} for any choice of $w_0 \in \mathcal{M}(-1, -1)$ yields the unique MPCC multiplier candidates
\begin{align}
\mu_0 &= 1, & \nu_0 &= \begin{pmatrix} -1 \\ 0 \end{pmatrix}, & \lambda_0 &= \begin{pmatrix} 0 \\ 0 \end{pmatrix}. \label{eq:Ex1MPCCMultipliers}
\end{align}
To analyze optimality we have to distinguish three  cases:
\ifTwoColumn
\begin{align*}
\begin{aligned}
&\textbf{Case 1:} \quad & w_0^{[1]} &= 1, & w_0^{[2]}&= 0 \\
&&\Rightarrow \quad \alpha &= \{ 1 \}, & \beta &= \{ 2 \}, & \gamma &= \emptyset \\
&\textbf{Case 2:} \quad & w_0^{[1]} &> 0, & w_0^{[2]}&> 0 \\
&&\Rightarrow \quad \alpha &= \{1, 2\}, & \beta &= \emptyset, & \gamma &= \emptyset \\
&\textbf{Case 3:} \quad & w_0^{[1]} &= 0, & w_0^{[2]}&= 1 \\
&&\Rightarrow \quad \alpha &= \{2\}, & \beta &= \{ 1 \}, & \gamma &= \emptyset
\end{aligned}
\end{align*}
\else
\begin{align*}
\begin{aligned}
&\textbf{Case 1:} \quad & w_0^{[1]} &= 1, & w_0^{[2]}&= 0 & &\Rightarrow & \alpha &= \{ 1 \}, & \beta &= \{ 2 \}, & \gamma &= \emptyset \\
&\textbf{Case 2:} \quad & w_0^{[1]} &> 0, & w_0^{[2]}&> 0 & &\Rightarrow & \alpha &= \{1, 2\}, & \beta &= \emptyset, & \gamma &= \emptyset \\
&\textbf{Case 3:} \quad & w_0^{[1]} &= 0, & w_0^{[2]}&= 1 & &\Rightarrow & \alpha &= \{2\}, & \beta &= \{ 1 \}, & \gamma &= \emptyset
\end{aligned}
\end{align*}
\fi
Case~2 is the easiest to analyze since the biactive set $\beta$ is empty. Hence, for any $w_0^{[1]} >0 $ and $w_0^{[2]} > 0$, the point $\*v = (-1, -1, w_0)$ is S-stationary and therefore a local optimum. For case~1, we see that the MPCC multipliers $\nu_0^{[2]}$ and $\lambda_0^{[2]}$ corresponding to the biactive complementarity constraint $2$ are both zero, hence, in this case we are also at an S-stationary local optimum. Case~3, on the other hand, does not satisfy the conditions for S-stationarity (or alternative sufficient optimality conditions, e.g.\ from~\cite{Ye2005,Guo2013}) since $\nu_0^{[1]} < 0$. From Theorem~\ref{thm:OptimalityConditions} it follows that $\*v^* = (-1, -1, 0, 1)$ is not locally optimal since it is only M-stationary and not S-stationary.

Figure~\ref{fig:Ex1ValueFunction} shows the cost function values $J(\*x, \*u)$ in~\eqref{eq:LCMPCCost} for varying values of $u_0$ (which is also the value function of~\eqref{eq:OCP}) together with the corresponding admissible values of $w_0^{[1]}$. Corresponding values for $w_0^{[2]}$ and $x_1$ follow from the dynamics~\eqref{eq:Ex1Dynamics}. The three cases discussed above are marked in green along the dashed red line indicating the feasible set of the optimal control problem.

\begin{figure}[thb]
\begin{center}
\includegraphics[width=.9\columnwidth]{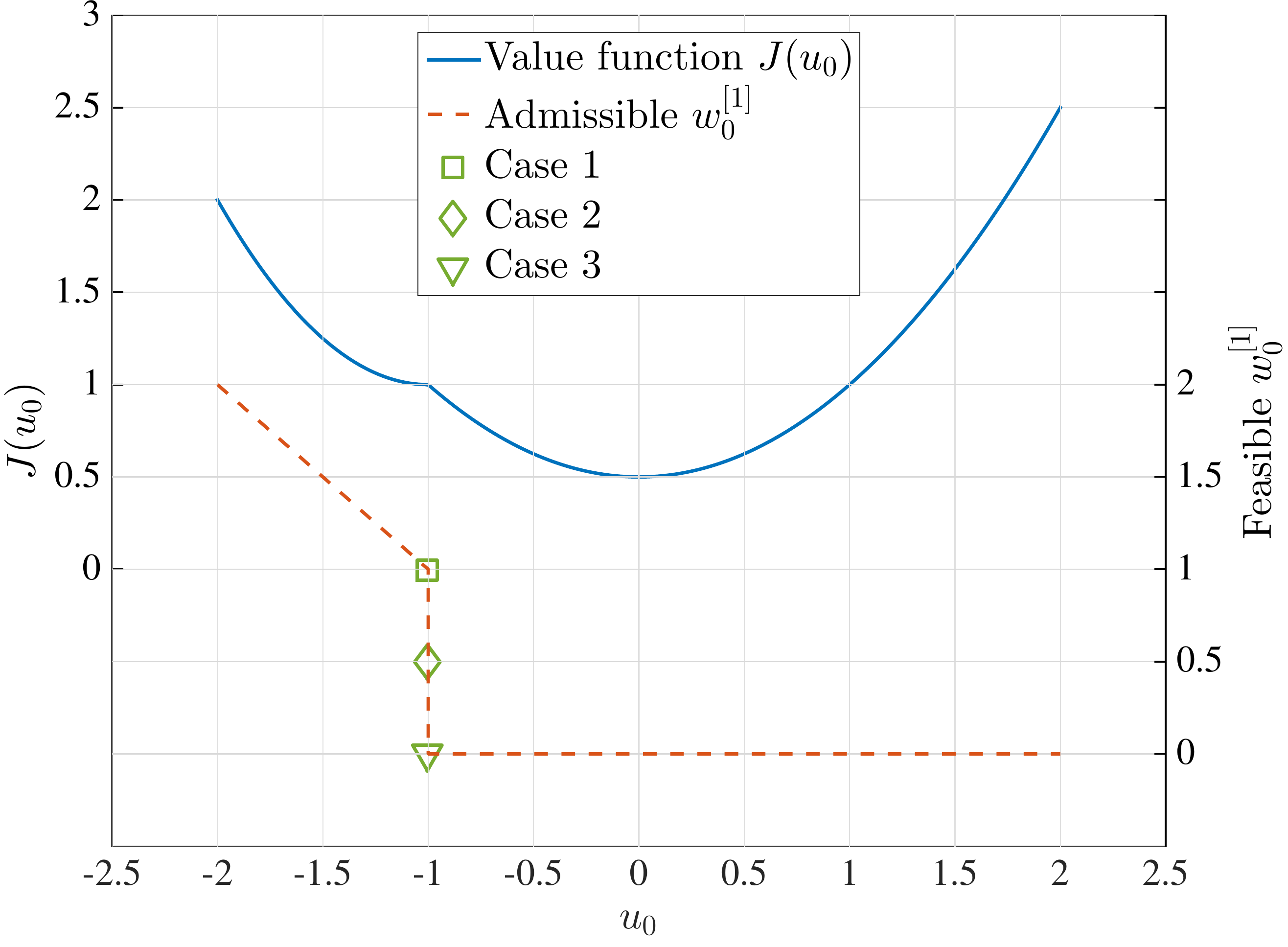}
\caption{Objective function value of optimal control problem~\eqref{eq:LCCFTOC} for example system~\eqref{eq:Ex1Dynamics} and varying control inputs together with admissible values for $w_0^{[1]}$.}
\label{fig:Ex1ValueFunction}
\end{center}
\end{figure}

From the value function $J$ in Figure~\ref{fig:Ex1ValueFunction} it is clear that a control input of $u_0 = -1$ should not be considered locally optimal from a control perspective and will not correspond to a local optimum of~\eqref{eq:OCP} in $\mathcal{U}$ because $\hat u = -1 + \epsilon$ for an arbitrarily small $\epsilon > 0$ is feasible and has a lower objective function value. At the same time, it illustrates how the local minima for~\eqref{eq:LCCFTOC} in cases~1 and~2 come about: in a small neighborhood around either of these points, a control input $u_0 > -1$ is not feasible because this would require a significant change in $w_0^{[1]}$. In other words, while
\[
\lVert (\*u^*, \*x^*) - (\*u, \*x) \rVert
\]
might be arbitrarily small,
\[
\lVert (\*u^*, \*x^*, \*w^*) - (\*u, \*x, \*w) \rVert \geq \epsilon \quad \forall \*w \in \mathcal{M}(\*u, \*x),
\]
for some finite $\epsilon >0$.
Hence, there exist values for $w_0$ such that $\*v^* = (u_0, x_1, w_0)$ is locally optimal for~\eqref{eq:LCCFTOC} with $u_0 = -1$. For $w_0^{[1]} = 0$ this suddenly changes because there exist tangential descent directions pointing inside the feasible set. At the same time, for any such $\*v^* \in \mathcal{Q}$ the corresponding point $(u_0, x_1) \in \mathcal{U}$ is not a local optimum for~\eqref{eq:OCP}.

\subsection{Necessary and sufficient conditions for optimal control trajectories}

The preceding example highlights the necessity of adapting Theorem~\ref{thm:OptimalityConditions} to the control setting. For control applications, one is primarily interested in the values of $\*u^*$ and $\*x^*$ and can disregard the actual value of $\*w^*$ (provided it is primal feasible). To that end, we introduce the notion of a \emph{locally optimal control trajectory}.

\begin{definition}[locally optimal control inputs] \label{def:OptimalTrajectories}
A  control input trajectory
$\*u^* := (u^*_0, u^*_1, \dots, u^*_{N-1})$ is called locally optimal with respect to~\eqref{eq:LCCFTOC} if there exists an $\epsilon > 0$
such that
\[
\ell_N(x^*_N) + \sum_{k=0}^{N-1} \ell_k(x^*_k, u^*_k) \leq \ell_N(x_N) + \sum_{k=0}^{N-1} \ell_k(x_k, u_k)
\]
for all control trajectories $\*u := (u_0, u_1, \dots, u_{N-1})$ with $\lVert \*u - \*u^* \rVert < \epsilon$. Here, $\*x^*$ and $\*x$ are the state trajectories resulting from applying control inputs $\*u^*$ and $\*u$, respectively.
\end{definition}

Obviously, a locally optimal control trajectory corresponds to local optima of~\eqref{eq:LCCFTOC} and~\eqref{eq:OCP}. The example in Section~\ref{subsec:Example} illustrates that the reverse does not hold: there may exist local optima for~\eqref{eq:LCCFTOC} that do not correspond to locally optimal control trajectories in~\eqref{eq:OCP}.
We immediately obtain necessary and sufficient conditions for a state and input trajectory $(\*x^*,\*u^*)$ to be locally optimal.

\begin{theorem}[optimality conditions for input trajectories] \label{thm:OptimalInputTrajectories}
Given an input trajectory $\*u^* \in \R^{N n_u}$ and its corresponding state trajectory $\*x^* \in \R^{(N+1) n_x}$ for an LC model~\eqref{eq:LCModel} satisfying Assumptions~\mbox{\ref{ass:WellPosedness}--\ref{ass:NontrivialSolutions}}, $\*u^*$ is locally optimal per Definition~\ref{def:OptimalTrajectories} if and only if $\*v^* := ( \*u^*, \*x^*, \*w)$ is S-stationary for all $\*w \in \mathcal{M}(\*u^*, \*x^*)$.
\end{theorem}
\begin{proof}
The input trajectory $\*u^*$ being locally optimal is equivalent to $\*v^*$ being a local optimum of~\eqref{eq:LCCFTOC} for all $\*w \in \mathcal{M}(\*u^*, \*x^*)$. The result follows from Theorem~\ref{thm:OptimalityConditions}.
\end{proof}

A number of sufficient conditions for an input trajectory $\*u^*$ to be locally optimal can be derived which are very easy to check.
The proofs of these results are  straightforward. The first Corollary mirrors the sufficient conditions in Theorem~\ref{thm:SufficientFirstOrderCondition} for a point to be globally optimal in~\eqref{eq:LCCFTOC}.

\begin{corollary}[globally optimal control trajectories] \label{cor:GloballyOptimalTrajectories}
Let $\*v^*=(\*u^*,\*x^*,\*w^*)$ be an S-stationary point for~\eqref{eq:LCCFTOC} and let the MPCC multipliers satisfy the conditions~\eqref{eq:SufficientOptimalityConditions} in Theorem~\ref{thm:SufficientFirstOrderCondition}. Then $\*u^*$ is a globally optimal input trajectory.
\end{corollary}

Specializing Theorem~\ref{thm:IsolatedMinimizer} to optimal control inputs yields the next result.
\begin{corollary}[isolated control trajectories are optimal] \label{cor:OptimalIsolatedTrajectories}
Let $\*v^*=(\*u^*,\*x^*,\*w^*)$ be an S-stationary point for~\eqref{eq:LCCFTOC} and let M-SSOSC hold at $\*v^*$. Then $\*u^*$ is an isolated locally optimal input trajectory.
\end{corollary}

While isolated local minimizers always have unique complementarity variables (see Theorem~\ref{thm:IsolatedMinimizer}) other (non-isolated) local minimizers can have the same property. Hence, from Definition~\ref{def:OptimalTrajectories} follows Corollary~\ref{cor:OptimalWithUniqueInnerMultipliers}.
\begin{corollary}[unique complementarity variables imply optimality] \label{cor:OptimalWithUniqueInnerMultipliers}
Let $\*v^*=(\*u^*,\*x^*,\*w^*)$ be an S-stationary solution for~\eqref{eq:LCCFTOC} and let $\mathcal{M}(\*u^*, \*x^*)$ be a singleton, i.e.\ $\*w^*$ is the unique complementarity variable trajectory consistent with $\*u^*$ and $\*x^*$. Then $\*u^*$ is a locally optimal input trajectory.
\end{corollary}

Theorem~3.1.7(b) in~\cite{Cottle1992} can be used to determine whether the given complementarity variable trajectory $\*w^*$ is unique using the determinant of a particular submatrix of $I_N \otimes E_w$. From Corollary~\ref{cor:StrictComplementarity} we know that $\beta = \emptyset$ is a necessary condition for $\mathcal{M}(\*u^*, \*x^*)$ to be a singleton.

\section{Inverse Optimization Modeling of Piecewise Affine Systems}
\label{sec:PWASystems}

In this section we will show how all the results from the previous sections can be applied to hybrid dynamical systems in continuous piecewise affine (PWA) form.
Their state dynamics are given as
\begin{align}
x^+ = f(x,u) =  A_i x + B_i u + c_i \quad \text{for} \quad (x,u) \in \Omega_i, \label{eq:PWAModel}
\end{align}
where $x \in \R^{n_x}$ is the system state and $u \in \R^{n_u}$ the control input. The $n_r$ regions $\Omega_i \subseteq \R^{n_x} \times \R^{n_u}$ form a partition of the domain $\Omega$ of the PWA system~\eqref{eq:PWAModel}, i.e.\ $\int(\Omega_i) \cap \int(\Omega_j) = \emptyset$ for $i \neq j$ and $\bigcup_{i=1}^{n_r} \Omega_i = \Omega$. We assume the system dynamics to be continuous across region boundaries, i.e.
\[
A_i x + B_i u + c_i = A_j x + B_j u + c_j
\]
for all $x \in \Omega_i \cap \Omega_j$. %

The finite horizon optimal control problem for PWA system~\eqref{eq:PWAModel} is given as
\begin{subequations}
\label{eq:PWACFTOC}
\begin{align}
\min_{\*u,\*x} \quad  &{\ell}_N(x_N) + \sum_{k=0}^{N-1} {\ell}_k(x_k,u_k) &\label{eq:PWAMPCCost} \\
\text{s.t.} \quad &\forall k = 0,\dots,{N-1} \colon
\ifTwoColumn \nonumber \\  \else & \fi
&x_{k+1} =  A_i x_k + B_i u_k + c_i \enspace \text{for} \enspace (x_k,u_k) \in \Omega_i, \label{eq:PWAMPCDynamicAssignment}
\end{align}
\end{subequations}
Here, $\*u := (u_0, \dots, u_{N-1})$ and $\*x := (x_0, \dots, x_{N})$ are the input and state trajectory from an initial state $x_0$ as in~\eqref{eq:LCCFTOC}.
These problems are typically solved with mixed-integer programming (MIP) approaches based on an MLD reformulation of the PWA dynamics~\cite{Bemporad1999}. Other solution approaches based on nonlinear programming~\cite{DeSchutter2001}, the solution of a series of linear programs~\cite{DeSchutter2004}, or the alternating direction method of multipliers~\cite{Frick2016} have also been proposed.

While the equivalence between  PWA models~\eqref{eq:PWAModel} and LC models~\eqref{eq:LCModel} has been known in the literature for quite some time, the LC models resulting from the derivations in~\cite{Heemels2001} will not satisfy Assumptions~\ref{ass:DecomposablePSDRankOne} and~\ref{ass:NontrivialSolutions} and require an MLD model as an intermediate step. The authors in~\cite{Heemels2001} point out that there will generally be a multitude of LC models~\eqref{eq:LCModel} corresponding to a given PWA model~\eqref{eq:PWAModel}. Which of these LC models is ``best'' will depend on the application.

Recent results from inverse optimization~\cite{Hempel2015a,Hempel2015} provide a direct link between PWA models~\eqref{eq:PWAModel} and LC models~\eqref{eq:LCModel}. This approach is based on the difference of convex functions~\cite{Horst1999} and in most cases of interest provides very compact LC models~\cite[Lem.~4]{Hempel2015a}. We can show that following this approach will lead to an LC model that satisfies Assumptions~\ref{ass:WellPosedness}--\ref{ass:NontrivialSolutions}.

We will later make use of the results in~\cite{Hempel2015} to represent the PWA system~\eqref{eq:PWAModel} as an optimizing process.
To this end we will require the following result on the representation of a continuous PWA function $\psi \colon \R^{\hat m} \to \R^{\hat n}$ with only convex component functions $\psi^{[i]} \colon \R^{\hat m} \to \R$ as the optimal solution to a parametric quadratic program:
\begin{subequations}
\begin{lemma}[convex PWA function as solution to PQP] \label{lem:CvxPWAFctAsPQP}
Let $\psi \colon \R^{\hat m} \to \R^{\hat n}$ be continuous PWA and such that $\psi^{[i]} \colon \R^{\hat m} \to \R$ is convex for all $i \in \Nn{\hat n}$. Then
\begin{align}
\psi(p) \in \arg \min_{y \in \R^{\hat n}} \enspace \frac{1}{2} \left\lVert y - \bar{\psi}(p) \right\rVert^2_Q \enspace \text{s.t.} \enspace y \geq \psi(p) \label{eq:PWAasPQPSolutionLemma}
\end{align}
for any diagonal matrix $Q \succ 0$ and an affine function $\bar{\psi} \colon \R^{\hat m} \to \R^{\hat n}$ with
\begin{equation}
\bar{\psi} (p) \leq \psi(p) \quad \forall p. \label{eq:UncOptimizerSupportAss}
\end{equation}
Furthermore, the $\arg \min$ is a singleton, i.e.\ $\psi(p)$ is the unique optimizer for~\eqref{eq:PWAasPQPSolutionLemma}.
\end{lemma}
\end{subequations}
\begin{proof}
Since $Q \succ 0$ is diagonal~\eqref{eq:PWAasPQPSolutionLemma} decouples into scalar optimization problems in $y^{[i]}$ as follows:
\ifTwoColumn
\begin{align*}
\psi^{[i]}(p) \in \arg \min_{y^{[i]} \in \R} \quad &\frac{1}{2} Q^{[i,i]} \left( y^{[i]} - \bar{\psi}^{[i]}(p) \right)^2 \\
 \text{s.t.} \quad &y^{[i]} \geq \psi^{[i]}(p)
\end{align*}
\else
\[
\psi^{[i]}(p) \in \arg \min_{y^{[i]} \in \R} \enspace \frac{1}{2} Q^{[i,i]} \left( y^{[i]} - \bar{\psi}^{[i]}(p) \right)^2 \enspace \text{s.t.} \enspace y^{[i]} \geq \psi^{[i]}(p)
\]
\fi
It is easy to see that this is simply the projection of $\bar{\psi}^{[i]}(p)$ onto the (convex) epigraph of $\psi^{[i]}(p)$. The result follows immediately from~\eqref{eq:UncOptimizerSupportAss}.
\end{proof}
Lemma~\ref{lem:CvxPWAFctAsPQP} is a variant of~\cite[Lem.~2]{Hempel2013} which considered a parametric \emph{linear} program in~\eqref{eq:PWAasPQPSolutionLemma}. Note that we can always find an affine  function satisfying~\eqref{eq:UncOptimizerSupportAss} due to the convex component functions of $\psi$, cf.~\cite[Thm.~8.13]{Rockafellar1998}.

We will adapt the procedure outlined in~\cite{Hempel2015} to derive a so-called \emph{inverse optimization model} for a given PWA system~\eqref{eq:PWAModel}.
To this end we use the fact that scalar-valued continuous PWA functions can be written as the difference of two convex PWA functions~\cite{Kripfganz1987,Hempel2013}. For the PWA dynamics~\eqref{eq:PWAModel} this means $f(x,u) = \psi(x,u) - \phi(x,u)$ where~
\begin{subequations}
 \label{eq:CvxConcaveParts}
\begin{align}
\psi(x,u) &= A_{y,j} x + B_{y,j} u + c_{y,j} & \text{for} \enspace (x,u) &\in \Omega_j^y, \label{eq:ConvexPart} \\
\phi(x,u) &= A_{z,k} x + B_{z,k} u + c_{z,k} & \text{for} \enspace (x,u) &\in \Omega_k^z \label{eq:ConcavePart}
\end{align}
\end{subequations}
are continuous PWA functions over the domain $\Omega$ of the PWA system and the component functions $\psi^{[i]}, \phi^{[i]} \colon \R^{n_x} \times \R^{n_u} \to \R$ are convex for all $i \in \Nn{n_x}$. The reader is referred to~\cite{Hempel2013,Hempel2015a} for details on the computation of the convex decomposition including the matrices $A_{y,j}$, $B_{y,j}$ et alia. Note that this \emph{convex decomposition} of $f$ is not unique. The regions $\Omega_j^y$ and $\Omega_k^z$ also form partitions of the domain for $j \in \Nn{n_r^y}$ and $k \in \Nn{n_r^z}$. They may differ from the original regions $\Omega_i$, unless the original regions form a so-called \emph{regular partition}~\cite[Lem.~2]{Hempel2015}. The reader is referred to~\cite{DeLoera2010} for a comprehensive treatment of regular partitions; typical examples are Delaunay triangulations, Voronoi diagrams, and related partitions.

We can apply Lemma~\ref{lem:CvxPWAFctAsPQP} to the PWA functions $\psi$ and $\phi$ in~\eqref{eq:CvxConcaveParts} that make up the PWA dynamics~\eqref{eq:PWAModel} to obtain the following inverse optimization model of~\eqref{eq:PWAModel}:
\begin{subequations}
\label{eq:IPQPModel}
\begin{align}
x^+  &= \hat y - \hat z \label{eq:IPQPTransition} \\
\hat y &\in \arg \min_{y \in \R^{n_x}} \enspace  \frac{1}{2} \left\lVert y - \bar{\psi}(x,u) \right\rVert^2_{Q_y} \label{eq:IPQPConvexCost} \\
&\text{s.t.} \enspace y \geq A_{y,j} x + B_{y,j} u + c_{y,j} \enspace \forall j \in \Nn{n_r^y}\label{eq:IPQPConvexConstraints} \\
\hat z &\in \arg \min_{z \in \R^{n_x}} \enspace  \frac{1}{2} \left\lVert z - \bar{\phi}(x,u) \right\rVert^2_{Q_z} \label{eq:IPQPConcaveCost} \\
&\text{s.t.} \enspace z \geq A_{z,k} x + B_{z,k} u + c_{z,k} \enspace \forall k \in \Nn{n_r^z}\label{eq:IPQPConcaveConstraints}
\end{align}
\end{subequations}
From Lemma~\ref{lem:CvxPWAFctAsPQP} it follows that $Q_y,Q_z \in \R^{n_x \times n_x}$ are arbitrary positive definite diagonal matrices. The affine functions $\bar{\psi}$ in~\eqref{eq:IPQPConvexCost} and $\bar{\phi}$ in~\eqref{eq:IPQPConcaveCost} are given as
\begin{align*}
\bar{\psi}(x,u) &= A_\psi x + B_\psi u + c_\psi, \\
\bar{\phi}(x,u) &= A_\phi x + B_\phi u + c_\phi,
\end{align*}
and must satisfy~\eqref{eq:UncOptimizerSupportAss} with respect to $\psi$ and $\phi$, respectively.

It follows directly from the construction that the inverse optimization model uses $\hat n = 2 n_x$ decision variables and a strictly convex cost function.
For alternative, possibly more compact, constructions of an inverse optimization model of~\eqref{eq:PWAModel} see~\cite{Hempel2012,Hempel2015,Nguyen2014}. An immediately obvious simplification can be performed in case a component function $f^{[i]}$ of the original PWA dynamics is already convex or concave.

Recall that~\eqref{eq:IPQPModel} represents a remodeling of the PWA dynamics in \eqref{eq:PWAMPCDynamicAssignment} in terms of a parametric optimization problem.    To facilitate the inclusion of this model into our overall optimal control problem~\eqref{eq:PWACFTOC}, we can rewrite~\eqref{eq:IPQPModel} in terms of its KKT conditions, yielding the following linear complementarity model for our system dynamics:
\begin{subequations}
\label{eq:PWAasTwoCvxPWAModels}
\ifTwoColumn
\begin{align}
x^+ - (y - z) &=   0 \label{eq:PWAsTwoCvxPWATransition} \\
Q_y \left( y -  \left(A_\psi x + B_\psi u + c_\psi \right) \right) - \sum_{i=1}^{n_r^y} \lambda_{i} &= 0\label{eq:PWACvxDynamicsStationarity} \\
Q_z \left( z - \left(A_\phi x + B_\phi u + c_\phi \right) \right) - \sum_{j=1}^{n_r^z} \theta_{j} &= 0\label{eq:PWAConcaveDynamicsStationarity} \\
\forall i \in \Nn{n_r^y} \colon ~
 0 \leq y - \left(A_{y,i} x + B_{y,i} u + c_{y,i} \right)   \perp  \lambda_{i} &\geq 0 \label{eq:PWACvxDynamicsComplementarity} \\
\forall j \in \Nn{{n_r^z}} \colon
 0 \leq z - \left(A_{z,j} x + B_{z,j} u + c_{z,j} \right)   \perp  \theta_{j} &\geq 0 \label{eq:PWAConcaveDynamicsComplementarity}
\end{align}
\else
\begin{align}
x^+ &=   y - z, \label{eq:PWAsTwoCvxPWATransition} \\
0 &= Q_y \left( y -  \left(A_\psi x + B_\psi u + c_\psi \right) \right) - \sum_{i=1}^{n_r^y} \lambda_{i} \label{eq:PWACvxDynamicsStationarity}, \\
0 &= Q_z \left( z - \left(A_\phi x + B_\phi u + c_\phi \right) \right) - \sum_{j=1}^{n_r^z} \theta_{j} \label{eq:PWAConcaveDynamicsStationarity}, \\
\forall i \in \Nn{n_r^y} \colon \enspace
 0 &\leq y - \left(A_{y,i} x + B_{y,i} u + c_{y,i} \right)  \enspace \perp \enspace \lambda_{i} \geq 0 \label{eq:PWACvxDynamicsComplementarity}, \\
\forall j \in \Nn{n_r^z} \colon \enspace
 0 &\leq z - \left(A_{z,j} x + B_{z,j} u + c_{z,j} \right)  \enspace \perp \enspace \theta_{j} \geq 0 \label{eq:PWAConcaveDynamicsComplementarity}.
\end{align}
\fi
\end{subequations}
To avoid confusion with the KKT multipliers and MPCC multipliers introduced in Section~\ref{sec:MPC}, we will call the $\lambda_i$ and $\theta_j$ in~\eqref{eq:PWAasTwoCvxPWAModels} \emph{internal multipliers}, i.e.\ multipliers for the lower-level part of the optimal control problem~\eqref{eq:PWACFTOC} when the dynamics have been reformulated as the KKT conditions of a parametric QP. They correspond to the complementarity variable $w$ in~\eqref{eq:LCModel}.

Note that the reformulation~\eqref{eq:PWAasTwoCvxPWAModels} of the PWA system~\eqref{eq:PWAModel} is an LC model~\eqref{eq:LCModel} with a generalized cone complementarity.
It can be shown that all results in Section~\ref{sec:MPC} also hold for the formulation~\eqref{eq:PWAasTwoCvxPWAModels} with the additional auxiliary variables $y$ and $z$~\cite{Hempel2016}. Alternatively, one could eliminate these auxiliary variables to obtain an equivalent, more compact complementarity representation that exactly matches~\eqref{eq:LCModel}:
\begin{subequations}
\label{eq:CompactLCModel}
\ifTwoColumn
\begin{align}
x^+ &= \left(A_\psi - A_\phi\right) x + \left( B_\psi - B_\phi \right) u + c_\psi - c_\phi \label{eq:CompactLCTransition} \\
&\qquad + Q_y^{-1} \sum_{i=1}^{n_r^y} \lambda_i - Q_z^{-1} \sum_{j=1}^{n_r^z} \theta_j , \nonumber \\
0 &\leq w_{\lambda_i} \enspace \perp \enspace \lambda_i \geq 0 \quad \forall i \in \Nn{n_r^y}, \label{eq:CompactComplY} \\
0 &\leq w_{\theta_j} \enspace \perp \enspace \theta_j \geq 0 \quad \forall j \in \Nn{n_r^z }, \label{eq:CompactComplZ}
\end{align}
where the complementarity variables $w_{\lambda_i}$ and $w_{\theta_j}$ are
\begin{align}
&\begin{aligned}
w_{\lambda_i} &= Q_y^{-1} \sum_{j=1}^{n_r^y}  \lambda_{j} + \left( A_\psi - A_{y,i}\right) x  \\
&\qquad + \left( B_\psi -  B_{y,i} \right) u + c_\psi - c_{y,i},
\end{aligned} \label{eq:CompactWY}  \\
&\begin{aligned}
w_{\theta_j} &=  Q_z^{-1} \sum_{i=1}^{n_r^z} \theta_i + \left( A_\phi - A_{z,j} \right) x \\
&\qquad + \left( B_\phi - B_{z,j} \right) u + c_\phi - c_{z,j}.
\end{aligned} \label{eq:CompactWZ}
\end{align}
\else
\begin{align}
x^+ &= \left(A_\psi - A_\phi\right) x + \left( B_\psi - B_\phi \right) u + c_\psi - c_\phi + Q_y^{-1} \sum_{i=1}^{n_r^y} \lambda_i - Q_z^{-1} \sum_{j=1}^{n_r^z} \theta_j \label{eq:CompactLCTransition} \\
w_{\lambda_i} &= Q_y^{-1} \sum_{j=1}^{n_r^y}  \lambda_{j} + \left( A_\psi - A_{y,i}\right) x + \left( B_\psi -  B_{y,i} \right) u + c_\psi - c_{y,i} \label{eq:CompactWY} \\
0 &\leq w_{\lambda_i} \enspace \perp \enspace \lambda_i \geq 0 \quad \forall i \in \Nn{n_r^y} \label{eq:CompactComplY} \\
w_{\theta_j} &=  Q_z^{-1} \sum_{i=1}^{n_r^z} \theta_i + \left( A_\phi - A_{z,j} \right) x + \left( B_\phi - B_{z,j} \right) u + c_\phi - c_{z,j} \label{eq:CompactWZ} \\
0 &\leq w_{\theta_j} \enspace \perp \enspace \theta_j \geq 0 \quad \forall j \in \Nn{n_r^z } \label{eq:CompactComplZ}
\end{align}
\fi
\end{subequations}
What is left to do is to show that it also satisfies the assumptions made at the start of the paper.
To that end, we make the following non-restrictive assumption on the functions $\bar \psi$ and $\bar \phi$:
\begin{subequations}
\label{eq:StrictAffineSupportCondition}
\begin{alignat}{2}
\bar{\psi}(x,u) &< \max_{j\in\Nn{n_r^y}} A_{y,j} x + B_{y,j} u + c_{y,j} &  \quad \forall (x,u) &\in \Omega \label{eq:StrictAffineSupportConditionCvx} \\
\bar{\phi}(x,u) &< \max_{i\in\Nn{n_r^z}} A_{z,i} x + B_{z,i} u + c_{z,i} & \forall (x,u) &\in \Omega \label{eq:StrictAffineSupportConditionConcave}
\end{alignat}
\end{subequations}
While~\eqref{eq:StrictAffineSupportCondition} is stronger than~\eqref{eq:UncOptimizerSupportAss}, it can always be satisfied for a given PWA system, e.g.\ by choosing
any index $ j \in \Nn{n_r^y}$ and $ i \in \Nn{n_r^z}$ and setting
\begin{subequations}
\label{eq:StrictAffineSupportConstruction}
\begin{align}
\bar{\psi}(x,u) &= A_{y, j} x + B_{y, j} u + c_{y, j} - \eta, \\
\bar{\phi}(x,u) &= A_{z, i} x + B_{z, i} u + c_{z, i} - \zeta,
\end{align}
\end{subequations}
with arbitrary $\eta,\zeta > 0$. With this we can prove that our results from Section~\ref{sec:MPC} also apply to properly remodeled PWA systems:

\begin{theorem} \label{thm:LCModelAssumptions}
Any LC model in the form~\eqref{eq:CompactLCModel} satisfying inequalities~\eqref{eq:StrictAffineSupportCondition} also satisfies Assumptions~\ref{ass:WellPosedness},~\ref{ass:DecomposablePSDRankOne}, and~\ref{ass:NontrivialSolutions}.
\end{theorem}
\begin{proof}
For a fixed $x \in \R^{n_x}$ and $u \in \R^{n_u}$ the conditions~\eqref{eq:PWACvxDynamicsStationarity} and~\eqref{eq:PWACvxDynamicsComplementarity} form the KKT system of the strictly convex PQP~\eqref{eq:IPQPConvexCost}-\eqref{eq:IPQPConvexConstraints} which has a unique minimizer, i.e.\ $\hat y$ is independent of the values of the internal multipliers $\lambda_i$. An analogous argument holds for the $\theta_j$ in~\eqref{eq:CompactWZ}-\eqref{eq:CompactComplZ} and $\hat z$ in~\eqref{eq:IPQPConcaveCost}-\eqref{eq:IPQPConcaveConstraints}, hence~\eqref{eq:CompactLCModel} satisfies Assumption~\ref{ass:WellPosedness}.

Careful examination of~\eqref{eq:CompactLCModel} reveals that the internal multipliers always appear as their sum over the regions of the convex decomposition. Accordingly, coupling between the internal multipliers occurs only for corresponding components. For any $i \in \Nn{n_x}$ we can collect all internal multipliers influencing that particular component of $x^+$ in the following LCP:
\ifTwoColumn
\[
\forall k \in \Nn{n_r^y} \colon \quad
0 \leq w^{[i]}_{k} \perp \lambda_k^{[i]} \geq 0,
\]
for which we define
\begin{align*}
w^{[i]}_{k} &:= \frac{1}{Q_y^{[i,i]}} \sum_{j=1}^{n_r^y} \lambda_j^{[i]} + \left( A_\psi^{[i,:]} - A_{y,k}^{[i,:]}\right) x \\
&\qquad + \left( B_\psi^{[i,:]} -  B_{y,k}^{[i,:]} \right) u + c_\psi^{[i]} - c_{y,k}^{[i]} .
\end{align*}
\else
\[
\forall k \in \Nn{n_r^y} \colon
0 \leq \frac{1}{Q_y^{[i,i]}} \sum_{j=1}^{n_r^y} \lambda_j^{[i]} + \left( A_\psi^{[i,:]} - A_{y,k}^{[i,:]}\right) x + \left( B_\psi^{[i,:]} -  B_{y,k}^{[i,:]} \right) u + c_\psi^{[i]} - c_{y,k}^{[i]} \perp \lambda_k^{[i]} \geq 0
\]
\fi
Collecting terms appropriately results in a complementarity sub-problem~\eqref{eq:SubComplementarityProblem} with
\begin{align*}
\ifTwoColumn
M_i &:= \frac{1}{Q_y^{[i,i]}} \one_{n_r^y} \one_{n_r^y}^\top, &
\else
M_i &:= \frac{1}{Q_y^{[i,i]}} \one_{n_r^y \times n_r^y} = \frac{1}{Q_y^{[i,i]}} \one_{n_r^y} \one_{n_r^y}^\top, &
\fi
C_i &:= \begin{pmatrix}A_\psi^{[i,:]} - A_{y,1}^{[i,:]} \\ \vdots \\ A_\psi^{[i,:]} - A_{y,n_r^y}^{[i,:]}\end{pmatrix}, \\
D_i &:= \begin{pmatrix}B_\psi^{[i,:]} - B_{y,1}^{[i,:]} \\ \vdots \\ B_\psi^{[i,:]} - B_{y,n_r^y}^{[i,:]}\end{pmatrix}, &
e_i &:= \begin{pmatrix}c_\psi^{[i]} - c_{y,1}^{[i]} \\ \vdots \\ c_\psi^{[i]} - c_{y,n_r^y}^{[i]}\end{pmatrix}.
\end{align*}
An analogous argument holds for the LCP in $\theta$, hence~\eqref{eq:CompactLCModel} satisfies Assumption~\ref{ass:DecomposablePSDRankOne}.

Finally, the component-wise satisfaction of~\eqref{eq:StrictAffineSupportCondition} guarantees that for every $i \in \Nn{n_x}$ there exists a $k \in \Nn{n_r^y}$ such that
\[
\left( A_\psi^{[i,:]} - A_{y,k}^{[i,:]}\right) x + \left( B_\psi^{[i,:]} -  B_{y,k}^{[i,:]} \right) u + c_\psi^{[i]} - c_{y,k}^{[i]} < 0.
\]
Again, we can show the same for the LCP in $\theta$ and have proven that~\eqref{eq:CompactLCModel} satisfies Assumption~\ref{ass:NontrivialSolutions}.
\end{proof}

Theorem~\ref{thm:LCModelAssumptions} finally proves our claim from the beginning of the paper --- every continuous PWA model~\eqref{eq:PWAModel} can be rewritten as an equivalent LC model~\eqref{eq:LCModel} satisfying Assumptions~\ref{ass:WellPosedness},~\ref{ass:DecomposablePSDRankOne}, and~\ref{ass:NontrivialSolutions}. Accordingly, such LC models are very general and cover a wide range of applications.
As will be seen in the next section, the slightly less compact but sparser formulation~\eqref{eq:PWAasTwoCvxPWAModels} can be advantageous for computational purposes. %

The example in Section~\ref{subsec:Example} illustrates that the possibly infinite number of admissible internal multipliers for a given $(x,u)$ can lead to spurious local minima in~\eqref{eq:LCCFTOC} that do not correspond to locally optimal control trajectories. It is easy to show that when writing a PWA system~\eqref{eq:PWAModel} as the LC model~\eqref{eq:CompactLCModel}, different consistent internal multipliers only exist for a given $(x,u) \in \Omega$ when there exists an $i \in \Nn{n_x}$ such that
\[
A^{[i,:]}_{y,j} x + B^{[i,:]}_{y,j} u + c^{[i]}_{y,j} = A^{[i,:]}_{y,k} x + B^{[i,:]}_{y,k} u + c^{[i]}_{y,k}
\]
for $j \neq k$ in~\eqref{eq:CvxConcaveParts} (or analogously for $z$). This is naturally the case when $(x,u)$ is on the boundary of two neighboring regions of the convex or concave part of the PWA dynamics. Accordingly, any control trajectory $\*u$ that lands on such a region boundary will correspond to a local minimum in~\eqref{eq:LCCFTOC} with $\mathcal{M}(\*u, \*x)$ not a singleton. It is conceivable that many of these local minima do not correspond to locally optimal control trajectories per Definition~\ref{def:OptimalTrajectories}. While the simulations below indicate that this problem is not severe it is an issue that warrants future research.

\section{Numerical Results}
\label{sec:Numerics}

We first present a small toy example adapted from~\cite{Hempel2013} that illustrates how PWA dynamics can be systematically transformed into an LC model. By construction, the resulting model will satisfy the conditions of Theorem~\ref{thm:LCModelAssumptions} and, hence, Theorem~\ref{thm:OptimalityConditions}. We then present extensive computational results on randomly generated PWA systems that show the possible benefits from reformulating PWA dynamics as described in Section~\ref{sec:PWASystems}. Using standard NLP solvers to solve the resulting MPCC~\eqref{eq:LCCFTOC} instead of dealing with a MIP formulation of~\eqref{eq:PWACFTOC} can be significantly faster without sacrificing much in terms of solution quality.

\subsection{Illustrative Example}

Consider the PWA dynamics $f$ shown in red in the middle of Figure~\ref{fig:2DExample} which represents the dynamics of a hypothetical PWA system~\eqref{eq:PWAModel} with one state $x$ and one control input $u$. Using the results from~\cite{Hempel2013,Hempel2015}, it can be decomposed into the two convex PWA functions $\psi$ and $\phi$ shown in the figure such that $f(x,u) = \psi(x,u) - \phi(x,u)$.

\begin{figure}[thb]
\begin{center}
\includegraphics[width=.9\columnwidth]{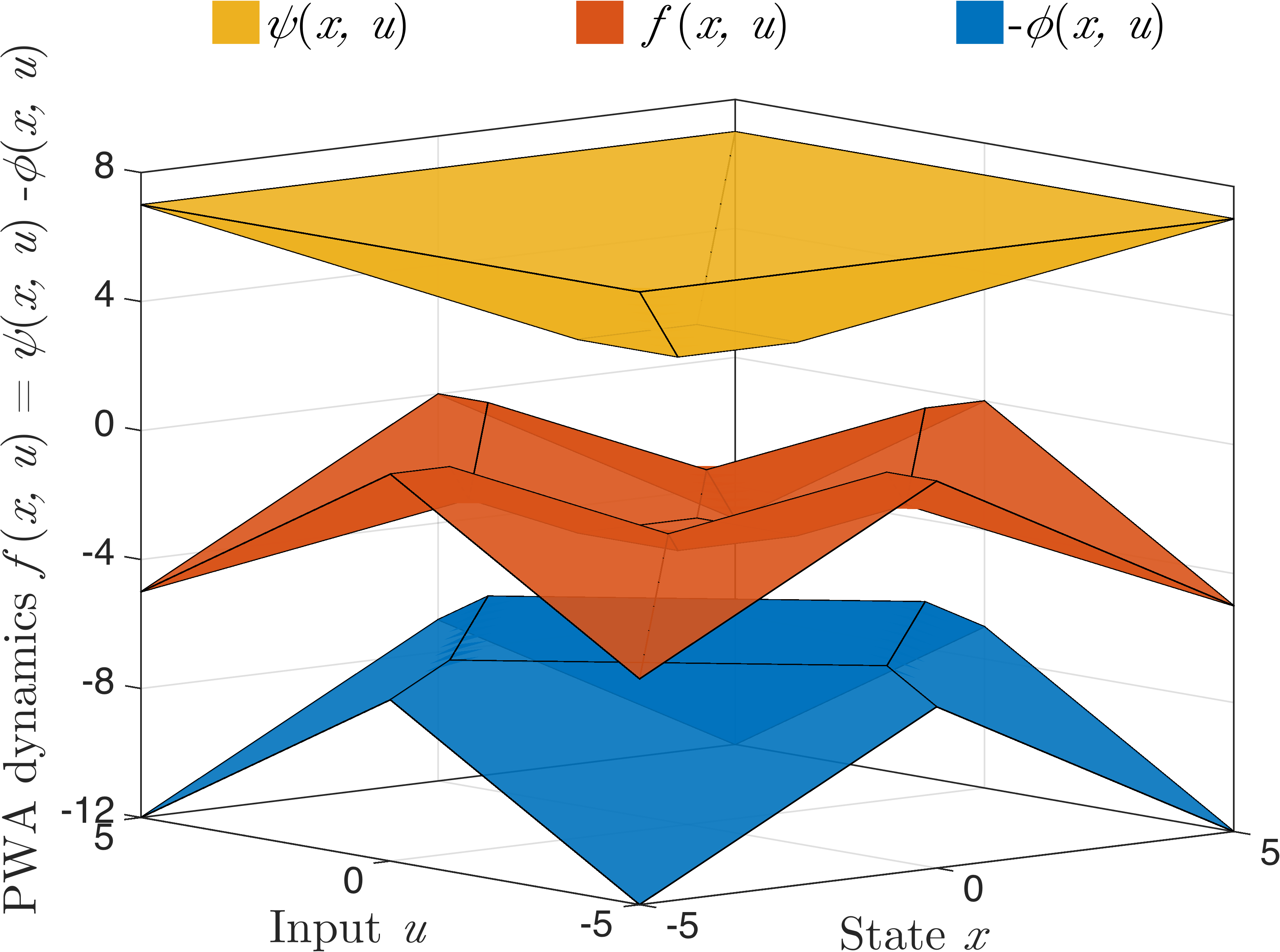}
\caption{Example of PWA system dynamics $f$ (in the middle) and its decomposition into a convex part $\psi$ (above) and a concave part $-\phi$ (below).}
\label{fig:2DExample}
\end{center}
\end{figure}

The explicit expressions for the functions $\psi \colon [-5,5] \times [-5,5] \to \R$ and $\phi \colon [-5,5] \times [-5,5] \to \R$ are given as follows:
\begin{align*}
\psi(x,u) &= \max \left\{ 3, \enspace x + 2, \enspace -x + 2, \enspace u + 2, \enspace -u + 2 \right\} \\
\phi(x,u) &= \max \left\{ 6,\enspace  x + u + 2, \enspace x - u + 2, \right. \\
&\qquad \qquad\,\, \left. -x - u+ 2, \enspace -x + u + 2 \right\}
\end{align*}
Their difference yields the given dynamics $f(\cdot) = \psi(\cdot) - \phi(\cdot)$ which are omitted due to space limitations.

Let us choose $\bar{\psi}(x,u) := 1$ and $\bar{\phi}(x,u) := x + u$, which satisfy~\eqref{eq:UncOptimizerSupportAss} with respect to $\psi$ and $\phi$, respectively. We can now use Lemma~\ref{lem:CvxPWAFctAsPQP} to construct the following  optimization problem in the two scalar decision variables $z_1$ and $z_2$ from the expressions for $\psi$ and $\phi$:
\begin{align}
  &\min_{z_1,z_2} \quad \frac{1}{2} \left(z_1 - 1 \right)^2 + \frac{1}{2} \left(z_2 - \left(x + u \right) \right)^2 \label{eq:ExampleOptimizingProcess} \\
  &\begin{aligned}
  \text{s.t.} \quad z_1 &\geq 3        & \quad z_2 &\geq 6               \nonumber \\
  z_1 &\geq \hphantom{-}x +2        & z_2 &\geq -x-u +2\nonumber \\
  z_1 &\geq -x +2                  & z_2 &\geq \hphantom{-}x-u+2  \nonumber \\
  z_1 &\geq \hphantom{-}u +2       & z_2 &\geq -x+u+2 \nonumber \\
  z_1 &\geq -u +2                  & z_2 &\geq \hphantom{-}x+u +2 \nonumber \\
  \end{aligned} \nonumber \\
  &\hspace{1.23cm} (x,u) \in [-5,5] \times [-5,5] \nonumber
\end{align}
Due to its separability we can obtain the explicit solution to this optimization problems as $z_1^*(x,u) = \psi(x,u)$ and $z_2^*(x,u) = \phi(x,u)$. By construction, the PWA dynamics $f$ shown in Figure~\ref{fig:2DExample} are recovered as $f(x,u) = z_1^*(x,u) - z_2^*(x,u) = \psi(x,u) - \phi(x,u)$. In other words, we have obtained an inverse optimization model~\eqref{eq:IPQPModel} for the PWA dynamics $f$.

It is straightforward to derive the LC model representations~\eqref{eq:PWAasTwoCvxPWAModels} and~\eqref{eq:CompactLCModel} from~\eqref{eq:ExampleOptimizingProcess} by following the derivations in Section~\ref{sec:PWASystems}. 
Because of our choice for $\bar \psi$ and $\bar \phi$, the resulting LC model will by construction satisfy the conditions of Theorem~\ref{thm:LCModelAssumptions}. Hence, instead of solving the optimal control problem~\eqref{eq:PWACFTOC} for the original PWA dynamics we can instead solve the MPCC~\eqref{eq:LCCFTOC} and be assured that Theorem~\ref{thm:OptimalityConditions} holds. Accordingly, we have a good chance to solve the MPCC problem with standard NLP solvers which try to find solutions to the KKT conditions.

\subsection{Computational Experiments}

To investigate the possible computational benefits gained over traditional approaches from using an LC model~\eqref{eq:LCModel} and solving the optimal control MPCC~\eqref{eq:LCCFTOC}, we randomly generated~10 different PWA models~\eqref{eq:PWAModel} with $n_x = 3$ and  $n_u = 1$.
The regions of these PWA models form a Delaunay-triangulation of the (bounded) system domain with 17 to 19 regions per system.   We used the Multi-Parametric Toolbox~3.0~\cite{Herceg2013} and ECOS~\cite{Domahidi2013} for all geometric computations necessary to generate the systems.

To solve~\eqref{eq:PWACFTOC} we used the MLD reformulation from~\cite{Bemporad1999} and modeled the resulting  MIP formulation with YALMIP~\cite{Yalmip}. For each prediction horizon $N \in \{2, 4, \dots, 10\}$ we generated~50 different initial states $x_0$ inside the domain of the system such that~\eqref{eq:PWACFTOC} was feasible, and solved the  mixed-integer optimal control problem with Gurobi~\cite{gurobiMendeley}. As the cost function we chose
\[
\ell_k(x_k, u_k) = \frac{1}{2} \left( \left\lVert x_k \right\rVert_2^2 + \left\lVert u_k \right\rVert_2^2 \right)\text{ and } \ell_N(x_N) = \frac{1}{2} \left\lVert x_N \right\rVert_2^2.
\]

Each of the generated PWA systems was transformed into both the general LC form~\eqref{eq:PWAasTwoCvxPWAModels} and the compact form~\eqref{eq:CompactLCModel}. The design parameters were chosen as $Q_y = Q_z = I$ and $\bar \psi$ and $\bar \phi$ by shifting parts of the convex decomposition downwards as shown in~\eqref{eq:StrictAffineSupportConstruction}.
The resulting formulations of~\eqref{eq:LCCFTOC} were solved from the same initial states $x_0$ using the general purpose NLP solver IPOPT~\cite{Wachter2006}. No special treatment of the complementarity constraints~\eqref{eq:LCMPCComplementarity} such as regularization~\cite{Scholtes2001} or penalization~\cite{Ralph2004} was performed; we simply implemented them as scalar bilinear inequalities.

\begin{figure}[thb]
\begin{center}
\input{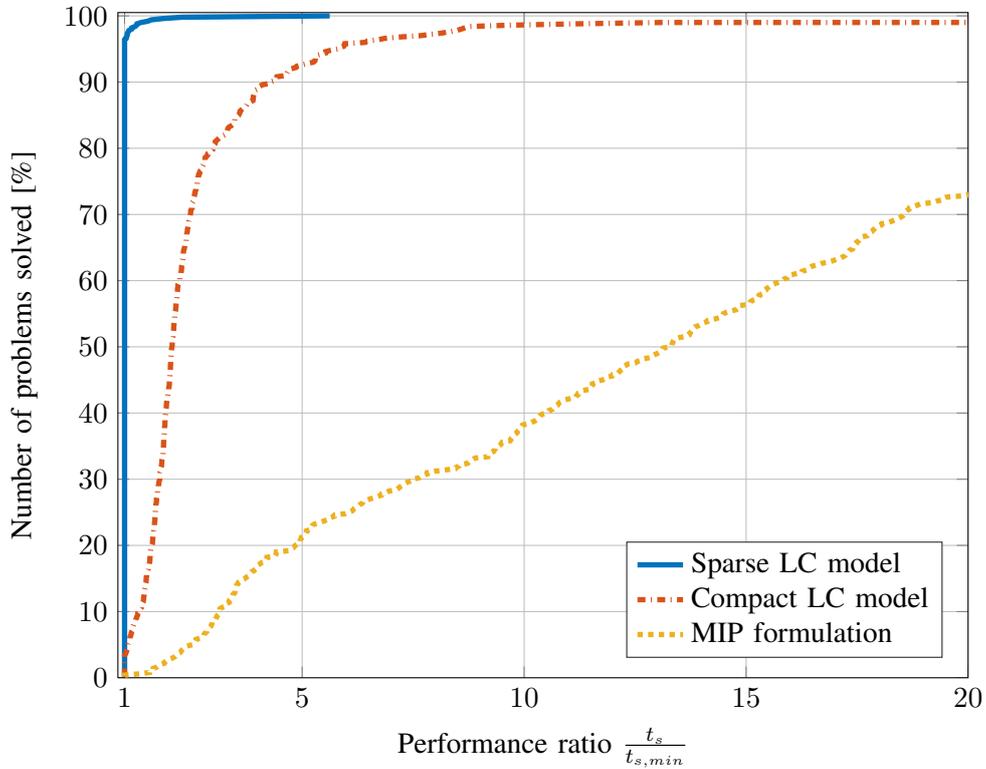}
\caption{Performance profile for $N=8$.}
\label{fig:PerformanceProfile}
\end{center}
\end{figure}

To get an impression of the performance that can be expected, Figure~\ref{fig:PerformanceProfile} shows a performance profile~\cite{Dolan2002} of the computation times required to solve the 500 problem instances for $N=8$. Define the performance ratio
\[
r_{p,s} := \frac{t_{p,s}}{\min_s t_{p,s}},
\]
where $t_{p,s}$ is the time it takes solver $s$ to solve problem instance $p$. The performance profile then plots the function $\rho_s \colon  \R \to [0, 1]$ defined as
\[
\rho_s(\tau) := \frac{1}{n_p} \left\lvert \{ p \mid r_{p,s} \leq \tau \} \right\rvert,
\]
which (for large numbers $n_p$ of problems $p$) is the probability that a performance ratio $r_{p,s}$ is within a factor of $\tau \in \R$ of the best possible ratio.

It can be seen from Figure~\ref{fig:PerformanceProfile} that in over 95~\% of all problem instances for $N=8$ the sparse LC formulation~\eqref{eq:PWAasTwoCvxPWAModels} was solved fastest. Additionally, for  every problem instance the MIP approach is (in terms of computation time) beaten by either the sparse LC~\eqref{eq:PWAasTwoCvxPWAModels} or the compact LC~\eqref{eq:CompactLCModel} approach. It is worth noting that such a consistent performance of a general purpose NLP algorithm such as IPOPT is a direct result of the special formulation of the optimal control problem~\eqref{eq:LCCFTOC} and Theorem~\ref{thm:OptimalityConditions}. It should not be expected when solving a general MPCC~\eqref{eq:MPCC} because the KKT conditions~\eqref{eq:KKTConditions} might not be satisfied at an optimum.

In addition to the often slower solution times for the compact formulation~\eqref{eq:CompactLCModel}, IPOPT encountered computational issues  for~9 out of~2500 problem instances, e.g.\ reaching the maximum number of iterations, numerical issues etc. The sparse formulation~\eqref{eq:PWAasTwoCvxPWAModels} had no such issues and solved all instances to local optimality.
While the sparse formulation~\eqref{eq:PWAasTwoCvxPWAModels} requires more decision variables in the optimal control problem~\eqref{eq:LCCFTOC} it also results in sparser constraint matrices whose structure may be more favorable to IPOPT. The issues of the compact formulation might be relieved by proper scaling or choice of the design parameters $Q_y$, $Q_z$, $\bar \psi$, and $\bar \phi$ in~\eqref{eq:IPQPModel}, but we will for the rest of this section only consider the sparse formulation~\eqref{eq:PWAasTwoCvxPWAModels}.

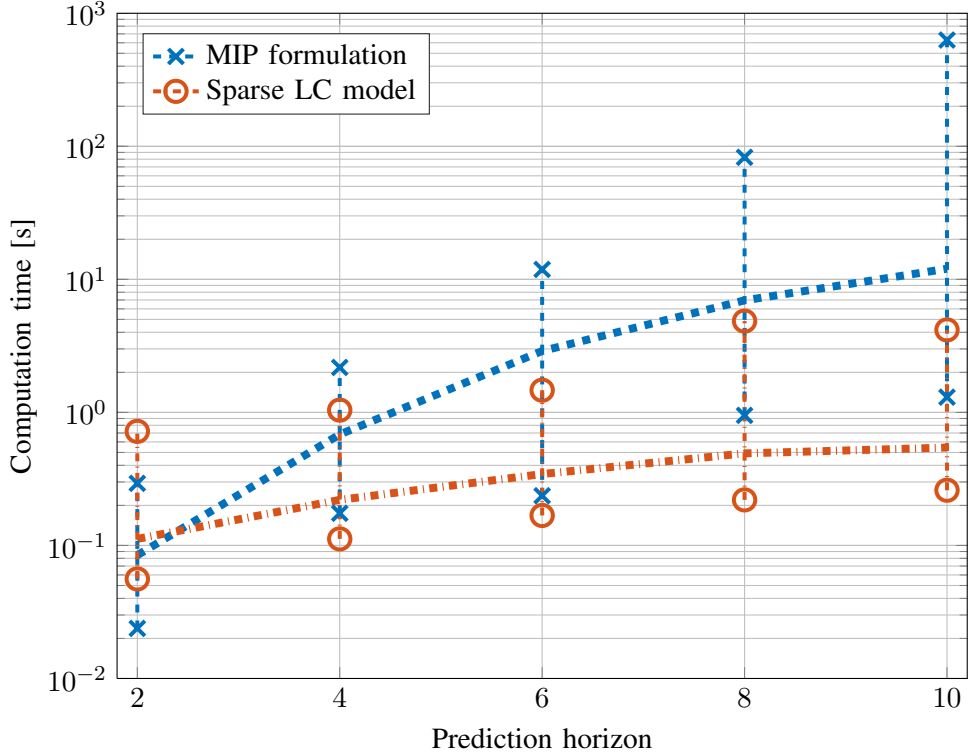
\begin{figure}[thb]
\begin{center}
%
\definecolor{mycolor1}{rgb}{0.00000,0.44700,0.74100}%
\definecolor{mycolor2}{rgb}{0.85000,0.32500,0.09800}%
\begin{tikzpicture}

\begin{axis}[%
width=0.761\columnwidth,
height=0.6\columnwidth,
at={(0\columnwidth,0\columnwidth)},
scale only axis,
xmin=1.8,
xmax=10.2,
xtick={ 2,  4,  6,  8, 10},
xlabel={Prediction horizon},
xmajorgrids,
ymode=log,
ymin=0.01,
ymax=1000,
yminorticks=true,
ylabel={Computation time [s]},
ymajorgrids,
yminorgrids,
axis background/.style={fill=white},
legend style={at={(0.03,0.97)},anchor=north west,legend cell align=left,align=left,draw=white!15!black}
]
\addplot [color=mycolor1,dashed,line width=3.0pt,forget plot]
  table[row sep=crcr]{%
2	0.0844265222549438\\
4	0.682844996452332\\
6	2.89228105545044\\
8	6.94967150688171\\
10	12.0162440538406\\
};
\addplot [color=mycolor2,dashdotted,line width=3.0pt,forget plot]
  table[row sep=crcr]{%
2	0.112006999999998\\
4	0.220013999999992\\
6	0.344021999999995\\
8	0.492030999999997\\
10	0.544034000000011\\
};
\addplot [color=mycolor1,dashed,line width=1.5pt,mark size=4.0pt,mark=x,mark options={solid},forget plot]
  table[row sep=crcr]{%
2	0.0237839221954346\\
2	0.292904138565063\\
};
\addplot [color=mycolor1,dashed,line width=1.5pt,mark size=4.0pt,mark=x,mark options={solid},forget plot]
  table[row sep=crcr]{%
4	0.174551010131836\\
4	2.17740297317505\\
};
\addplot [color=mycolor1,dashed,line width=1.5pt,mark size=4.0pt,mark=x,mark options={solid},forget plot]
  table[row sep=crcr]{%
6	0.236302852630615\\
6	11.8742837905884\\
};
\addplot [color=mycolor1,dashed,line width=1.5pt,mark size=4.0pt,mark=x,mark options={solid},forget plot]
  table[row sep=crcr]{%
8	0.949717044830322\\
8	82.6818161010742\\
};
\addplot [color=mycolor1,dashed,line width=1.5pt,mark size=4.0pt,mark=x,mark options={solid}]
  table[row sep=crcr]{%
10	1.30181503295898\\
10	629.060441970825\\
};
\addlegendentry{MIP formulation};

\addplot [color=mycolor2,dashdotted,line width=1.5pt,mark size=4.0pt,mark=o,mark options={solid},forget plot]
  table[row sep=crcr]{%
2	0.0560039999999979\\
2	0.720044999999999\\
};
\addplot [color=mycolor2,dashdotted,line width=1.5pt,mark size=4.0pt,mark=o,mark options={solid},forget plot]
  table[row sep=crcr]{%
4	0.112006999999991\\
4	1.040065\\
};
\addplot [color=mycolor2,dashdotted,line width=1.5pt,mark size=4.0pt,mark=o,mark options={solid},forget plot]
  table[row sep=crcr]{%
6	0.168010999999979\\
6	1.46809200000001\\
};
\addplot [color=mycolor2,dashdotted,line width=1.5pt,mark size=4.0pt,mark=o,mark options={solid},forget plot]
  table[row sep=crcr]{%
8	0.220014000000106\\
8	4.85630299999991\\
};
\addplot [color=mycolor2,dashdotted,line width=1.5pt,mark size=4.0pt,mark=o,mark options={solid}]
  table[row sep=crcr]{%
10	0.260015999999951\\
10	4.16826100000003\\
};
\addlegendentry{Sparse LC model};

\end{axis}
\end{tikzpicture}%
\caption{Computation times for different prediction horizons $N$. The center line is the median computation time over 500 problem instances for each prediction horizon and the error bars indicate the best and worst computation times, respectively.}
\label{fig:CompTimes}
\end{center}
\end{figure}
Figure~\ref{fig:CompTimes} shows the best-case, worst-case, and median computation times necessary to solve the 500 problem instances for each prediction horizon for the sparse, general LC model~\eqref{eq:PWAasTwoCvxPWAModels} and the PWA model~\eqref{eq:PWAModel}. %
The LC formulation~\eqref{eq:PWAasTwoCvxPWAModels} shows both much shorter computation times as well as gentler scaling for longer prediction horizons, e.g.\ the worst-case computation time for $N=10$ is two orders of magnitude larger for the MIP formulation.

To evaluate solution quality we can consider the relative difference between the achieved objective function values
\[
\frac{J_{NLP}^* - J_{MIP}^*}{J_{MIP}^*} \times 100~\%.
\]
For 99.2~\% of the 2500 problem instances over all prediction horizons this quantity is lower than~10~\% and in~71.9~\% of the cases the inverse optimization approach even yields a globally optimal solution.

These results are  affected by the choice of design parameters when constructing the inverse optimization model~\eqref{eq:IPQPModel}. For the results shown here we chose $\bar \psi$ and $\bar \phi$ as shown in~\eqref{eq:StrictAffineSupportConstruction} with $\eta = \xi = 0.5 \cdot\one_{}$, while with $\eta = \xi = 10 \cdot \one_{}$ less than 50~\% of the problem instances are solved to global optimality. Referring back to our example in Section~\ref{subsec:Example}, the value of $\eta$ corresponds to the (invariant) value of the sum of the internal multipliers in~\eqref{eq:Example1InnerMultipliers}. Hence, a larger value of $\eta$ corresponds to a larger set (in volume) of equivalent internal multipliers, i.e.\ cases 1 and 3 in Figure~\ref{fig:Ex1ValueFunction} move further apart.

For 187 of the 249 problem instances whose optimal values $J^*_{NLP}$ are further than 1~\% away from the globally optimal $J^*_{MIP}$, IPOPT returns control input trajectories $\*u_{NLP}$ which for at least one time step land exactly on the boundary between two regions of the PWA system. The other 62 problem instances correspond to locally optimal input trajectories in the sense of Definition~\ref{def:OptimalTrajectories} that traverse the interior of the PWA system's regions. To verify whether the boundary cases correspond to locally optimal trajectories we could use Theorem~\ref{thm:OptimalInputTrajectories} (which requires a vertex enumeration for the set $\mathcal{M}(\*u_{NLP}, \*x_{NLP})$) or one of its Corollaries. Due to the relatively short computation times for the optimal control MPCC it may be a viable strategy to solve~\eqref{eq:LCCFTOC} from different initial points and use the best result. This would reduce the chance of obtaining a solution on a region boundary and could even be implemented in parallel if the hardware allowed this.

Using the LC formulation~\eqref{eq:PWAasTwoCvxPWAModels} instead of an MIP approach to the PWA model~\eqref{eq:PWAModel} presents a trade-off where a significantly shorter and more consistent solution time is bought with a potential decrease in solution quality. Simulation evidence indicates that this degradation is negligible and its severity within reasonable bounds in most cases,
in particular because the optimal control problem~\eqref{eq:LCCFTOC} is typically solved in a receding horizon setting where only the first step of the control input trajectory $\* u^*$ would be applied before re-solving the problem at the next time step~\cite{Rawlings2009}.

\section{Conclusion}

In this paper we consider constrained optimal control problems for hybrid dynamical systems of linear complementarity type satisfying a number of structural assumptions. The optimal control problems are then mathematical programs with complementarity constraints which can be modeled as continuous nonlinear programs.
Under the assumptions made, we can prove that the classical Karush-Kuhn-Tucker conditions are necessary and sufficient for optimality which is rarely the case for general MPCCs.

Additionally, it is shown how continuous piecewise-affine systems can always be written as LC models satisfying our initial assumptions. This enables the treatment of control problems for a large and important class of hybrid systems as continuous NLPs using standard solution software. Numerical simulations illustrate the efficacy of this approach where the NLP can be solved in significantly shorter and more consistent time than a more traditional mixed-integer approach. The downside is a possible degradation in solution quality although this is often negligible since the NLP approach finds the global optimum in many cases. %

It should be investigated how much can be gained from using a solution method tailored to MPCCs. A number of relaxation and penalization methods have been suggested in the literature, cf.~\cite{Hoheisel2011} for a recent overview, but they have to accommodate the fact that M-stationarity is the strongest optimality condition available for most MPCCs. A straightforward nonlinear programming approach to~\eqref{eq:LCCFTOC} can prove successful due to the strong stationarity results in this paper. Alternatively, an MPCC-method that guarantees convergence to S-stationary solutions could be used, cf.~\cite{Izmailov2012}.

\iftoggle{useBiber}{
\printbibliography
~
}{
\bibliographystyle{IEEEtran}
\bibliography{bibliography.bib}
}

\balance

\end{document}